\documentclass{amsart}
\usepackage{amsmath,amsfonts,amsthm,amssymb,indentfirst,epic,url,graphics}

\newtheorem{theorem}{Theorem}
\newtheorem{lemma}[theorem]{Lemma}
\newtheorem{corollary}[theorem]{Corollary}
\newtheorem{proposition}[theorem]{Proposition}
\newtheorem{definition}[theorem]{Definition}

\newtheorem{question}[theorem]{Question}
\newtheorem{example}[theorem]{Example}

\newcommand{\U}{\mathcal{U}\kern.05em}
\newcommand{\F}{\mathcal{F}}
\newcommand{\Z}{\mathbb{Z}}
\newcommand{\Q}{\mathbb{Q}}
\renewcommand{\r}{\mathrm}
\newcommand{\card}{\r{card}}
\newcommand{\supp}{\r{supp}}

\newcommand{\lang}{\begin{picture}(5,7)
\put(1.2,2.5){\rotatebox{45}{\line(1,0){6.0}}}
\put(1.2,2.5){\rotatebox{315}{\line(1,0){6.0}}}
\end{picture}\kern.16em}
\newcommand{\rang}{\kern.1em\begin{picture}(5,7)
\put(.1,2.5){\rotatebox{135}{\line(1,0){6.0}}}
\put(.1,2.5){\rotatebox{225}{\line(1,0){6.0}}}
\end{picture}}

\begin{document}

\begin{center}
\texttt{Comments, corrections,
and related references welcomed, as always!}\\[.5em]
{\TeX}ed \today
\vspace{2em}
\end{center}

\title%
[Homomorphisms on infinite direct products]%
{Homomorphisms on infinite direct products\\
of groups, rings and monoids}
\thanks{Archived at \url{http://arXiv.org/abs/1406.1932}\,.
After publication, any updates, errata, related references,
etc., found will be recorded at
\url{http://math.berkeley.edu/~gbergman/papers/}
}

\subjclass[2010]{Primary: 03C20, 08B25, 17A01, 20A15, 20K25, 20M15,
Secondary: 16B70, 16P60, 20K40, 22B05.}
\keywords{homomorphism on an infinite direct product of groups,
rings, or monoids; ultraproduct; slender, algebraically compact,
and cotorsion abelian groups.
}

\author{George M. Bergman}
\address{University of California\\
Berkeley, CA 94720-3840, USA}
\email{gbergman@math.berkeley.edu}

\begin{abstract}
We study properties of a group, abelian group, ring, or monoid $B$
which (a)~guarantee that every homomorphism from an infinite
direct product $\prod_I A_i$ of objects of the same sort onto $B$
factors through the direct product of finitely many ultraproducts of
the $A_i$ (possibly after composition with the natural map
$B\to B/Z(B)$ or some variant), and/or
(b)~guarantee that when a map does so factor (and the
index set has reasonable cardinality), the
ultrafilters involved must be principal.

A number of open questions and topics for further investigation
are noted.
\end{abstract}
\maketitle

\section{Introduction}\label{S.intro}
A direct product $\prod_{i\in I} A_i$ of infinitely many
nontrivial algebraic structures is in general a ``big'' object:
it has at least continuum cardinality, and if the operations
of the $A_i$ include a vector-space structure,
it has at least continuum dimension.
But there are many situations where the set
of homomorphisms from such a product to a fixed
object $B$ is unexpectedly restricted.

The poster child for this phenomenon is
the case where the objects are abelian groups,
and $B$ is the infinite cyclic group.
In that situation, if the index set $I$ is countable (or, indeed, of
less than an enormous cardinality -- some details are recalled
in~\S\ref{S.gp_prin}),
then every homomorphism $\prod_{i\in I} A_i\to B$ factors through
the projection of $\prod_{i\in I} A_i$ onto the product of finitely
many of the $A_i.$
An abelian group $B$ which, like the infinite cyclic group, has this
property, is called ``slender''.
Slender groups have been completely characterized~\cite{Nunke},
and slender modules over general rings have been studied.

Recent work by N.\,Nahlus and the author (\cite{prod_Lie1},
\cite{prod_Lie2}, \cite{cap_ultra}) on
factorization properties of homomorphisms on infinite direct
products of not-necessarily-associative algebras (motivated
by the case of Lie algebras) have turned up interesting
variants on the above sort of behavior.

First, it turns out that in that context, a useful way to prove
every surjective homomorphism
$\prod_{i\in I} A_i\to B$ factors through
finitely many of the $A_i$ is by proving
(a)~that every such homomorphism factors
through the product of finitely many {\em ultraproducts}
of the $A_i,$ and also
(b)~that whenever one has a map that factors
in that way, the ultrafilters involved must be principal.
In this note, we shall consider each of conditions~(a) and~(b) on
an object $B$ as of separate interest.

Secondly, we found in \cite{prod_Lie1},
\cite{prod_Lie2}, \cite{cap_ultra} that in many
cases, though one cannot say that
every surjective homomorphism from a direct product to $B$ will itself
factor in one of these ways, one can say
that for every such homomorphism $\prod_{i\in I} A_i\to B,$ the
induced homomorphism $\prod_{i\in I} A_i\to B/Z(B)$ so
factors, where $Z(B)$ denotes the zero-multiplication ideal,
$\{b\in B\mid b\,B=B\,b=\{0\}\}$
(which for $B$ a Lie algebra is the center of $B).$
In the next section, we shall get similar results
for groups, with $Z(B)$ the center of the group $B.$
(Note that these statements do not say that
every surjective homomorphism $\prod_{i\in I} A_i\to B/Z(B)$
factors as stated;
such a factorization is asserted only when the
homomorphism $\prod_{i\in I} A_i\to B/Z(B)$ can be lifted to a
homomorphism $\prod_{i\in I} A_i\to B.)$
Maalouf \cite{FM} abstracts this property, and
strengthens some of the results of the papers cited.

In the classical case of abelian
groups (and its generalization to modules),
the condition on an object $B$ that every homomorphism from an
infinite product {\em onto} $B$ yield a factorization through
finitely many of the $A_i,$ and the corresponding
condition for homomorphisms {\em into} $B,$ are equivalent.
Indeed, from any homomorphism
$\prod_{i\in I} A_i\to B,$ one can get, in an obvious way, a surjective
homomorphism $B\times\prod_{i\in I} A_i\to B,$ and the original
homomorphism factors through finitely many of the $A_i$ if and only if
that surjective map factors through $B$ and finitely many $A_i.$
This observation uses implicitly the fact that one can add
homomorphisms of abelian groups -- in this case, the map
$B\times\prod_{i\in I} A_i\to B$ induced by the given map
on the one hand, and the projection to $B$ on the other.
But one cannot do this for homomorphisms of
noncommutative groups, of algebras, etc.;
so for these, the condition involving arbitrary maps and
the condition involving surjective maps are not equivalent.
In these cases, the condition on $B$ defined
in terms of surjective homomorphisms is the more informative.
Once one has characterized those $B$ for which all surjective
homomorphisms $\prod_{i\in I} A_i\to B$ yield such a factorization,
one can, if one wishes, characterize the $B$ with the
corresponding property for general homomorphisms as the
objects all of whose subobjects have the previous property.

In stating results of the sort we shall obtain, one has a choice
between (i)~saying that if a structure $B$ does {\em not}
have one or another of a list of ``messy'' properties, then
every homomorphism from an infinite direct product
onto $B$ leads to a certain kind of factorization, or (ii)~the
contrapositive statement, that if there exists a
homomorphism onto $B$ that does not so factor, then $B$ has
one of those messy properties.
Each approach has its plusses and minuses; here I have
followed~(ii), because it seems more straightforward to
understand how a non-factorable map forces $B$ to have a messy property
than to show that the absence of certain messy properties
implies that all maps factor; and also because some of the
conditions on $B$ come in several versions, and I
find it easier to parse a statement having a single hypothesis
and several conclusions than one with several alternative
hypotheses giving a single conclusion.
(But the above choice also has its awkward aspects;
I can't say which is really best.)

In \S\S\ref{S.gp_ultra}-\ref{S.gp_prin}, we shall
study the case where our
structures are not-necessarily-abelian groups, in
\S\S\ref{S.ab}-\ref{S.ab_Chase&&},
abelian groups, then, briefly, in \S\ref{S.rings}
and \S\ref{S.monoids}, rings and monoids.
In \S\ref{S.lattices} we note why lattices
are likely to be another case worth examining.

For a short review, for the nonspecialist, of the concepts of filter,
ultrafilter and ultraproduct, see \cite[Appendix~A]{prod_Lie1};
and for measurable cardinals $\kappa,$ and $\!\kappa\!$-complete
ultrafilters, which come up in
\S\S\ref{S.gp_prin}-\ref{S.ab} below, \cite[Appendix~B]{prod_Lie1}.
For detailed developments of these
concepts see, e.g., \cite{Ch+Keis} or \cite{C+N}.

We remark that there is in the
literature a concept of ``noncommutative slender group''
that is quite different from the subject of
\S\S\ref{S.gp_ultra}-\ref{S.gp_prin} below.
The concept so named can be arrived at by regarding the infinite
direct product in the definition of a slender abelian group
as a {\em completed direct sum}, and using in the noncommutative case,
instead of the direct product, an analogously completed
noncommutative coproduct.
For work on that topic
see~\cite{ncmSpecker} and references given there.

\section{Factoring group homomorphisms through finitely many ultraproducts.}\label{S.gp_ultra}

Let $(G_i)_{i\in I}$ be a family of groups.
By the {\em support} of an element
$g=(g_i)_{i\in I}\in\prod_{i\in I} G_i,$ we will understand the set
\begin{equation}\begin{minipage}[c]{35pc}\label{d.supp}
$\supp(g)\ =\ \{i\in I\mid g_i\neq e\}\ \subseteq\ I.$
\end{minipage}\end{equation}
Given any subset $S\subseteq I,$ we shall
identify $\prod_{i\in S} G_i$ in the obvious way with
the subgroup of $\prod_{i\in I} G_i$ consisting of elements
whose support is contained in $S.$
In particular, for $g\in\prod_{i\in I} G_i,$ the statement
$g\in\prod_{i\in S}G_i$ will mean $\supp(g)\subseteq S,$
and the statement $g\in G_i$ will mean $\supp(g)\subseteq\{i\}.$

Whereas the theory of slender abelian groups is based on delicate
structural properties of those groups, most of
our results on nonabelian groups will be based on a much
simpler observation: Elements of $\prod_{i\in I} G_i$ with disjoint
supports centralize one another.
As a quick example, it is not hard to see that
if $B$ is a simple nonabelian group, and we have any surjective
homomorphism $f: \prod_{i\in I} G_i\to B,$ then for each
$S\subseteq I,$ the map $f$ must annihilate one of the mutually
centralizing subgroups
$\prod_{i\in S} G_i$ and $\prod_{i\in S-I} G_i.$
From this one can deduce that the subsets $S\subseteq I$ such that
$f$ factors through the projection
$\prod_{i\in I} G_i\to\prod_{i\in S} G_i$ form an ultrafilter
(principal or nonprincipal) on $I.$

In the opposite direction, however, if we take for $B$
a cyclic group of prime order $p$ (thus losing the leverage
provided by noncommutativity), and
let all the $G_i$ be copies of that group,
then by linear algebra over the field of $p$ elements, there
exist homomorphisms $\prod_{i\in I} G_i\to B$
that send every $G_i$ onto $B,$ and hence don't factor through
any proper subproduct $\prod_{i\in S} G_i.$

As indicated in the Introduction, we shall get around the problem
created (as above) by commutativity by
composing homomorphisms $\prod_{i\in I} G_i\to B$ with the
quotient map $B\to B/Z(B),$
where $Z(B)$ is the center of $B.$
Given a homomorphism $f:\prod_{i\in I} G_i\to B,$
the key to our considerations will be the family of subsets
\begin{equation}\begin{minipage}[c]{35pc}\label{d.F}
$\F\ =\ \{S\subseteq I\mid$ the composite map
$\prod_{i\in I} G_i\to B \to B/Z(B)$\\
\hspace*{13em} factors through the
projection $\prod_{i\in I} G_i\to \prod_{i\in S} G_i\}$\\[.3em]
\hspace*{1.5em}%
$=\ \{S\subseteq I\mid f(\prod_{i\in I-S} G_i)\subseteq Z(B)\}.$
\end{minipage}\end{equation}
It is easy to see that $\F,$ so defined, is a {\em filter} on $I,$
and that if we write
\begin{equation}\begin{minipage}[c]{35pc}\label{d.pi}
$\pi: B\ \to\ B/Z(B)$
\end{minipage}\end{equation}
for the quotient map, then $\F$ is the
largest filter such that $\pi f:\prod_{i\in I} G_i\to B/Z(B)$
factors through the reduced product $\prod_{i\in I} G_i/\F.$
(The above observation, and the next few, do not yet use the fact that
we are working with a map of the form $\pi f,$ but only that we are
considering a homomorphism on a product group.
The fact that our map has the form $\pi f$
will become significant starting with Lemma~\ref{L.gp_Us} below.)

If the filter $\F$ of~\eqref{d.F}
is a finite intersection of distinct ultrafilters,
$\U_0\cap\dots\cap\U_{n-1},$ then
$\prod_{i\in I} G_i/\F\cong\prod_{i\in I} G_i/\U_0\times\dots\times
\prod_{i\in I} G_i/\U_{n-1},$
so $\pi f$ factors through the projection to that product; and
conversely, if
$\pi f$ factors through the projection to such a product,
then $\F$ is the intersection of some subset of
the $\U_k$ (the minimal set of $\U_k$ allowing such a factorization).
In this connection, we recall
\begin{equation}\begin{minipage}[c]{35pc}\label{d.cap_ultra}
(\cite[Lemma~3, $(3){\iff}(5)$]{cap_ultra})
A filter $\F$ on a set $I$ can be written as
the intersection of finitely many ultrafilters on $I$ if and only if
for every partition of $I$ into countably many sets $J_m$
$(m\in\omega),$ there is at least one $m\in\omega$ such
that $I-J_m\in\F.$
\end{minipage}\end{equation}
Here and below, we make the conventions that a partition may include
one or more instances of the empty set, and that the
intersection of the empty family of filters on a set is the
set of all subsets of that set, i.e., the improper filter.
(These conventions are needed to make various statements correct
in degenerate cases.)

Let us note what~\eqref{d.cap_ultra} tells us about homomorphisms
on direct product groups.

\begin{lemma}\label{L.gp_Us}
Let $f:\prod_{i\in I} G_i\to B$ be a homomorphism from a direct
product of
groups $G_i$ to a group $B,$ which is surjective; or more generally,
such that the composite $\pi f:\prod_{i\in I} G_i\to B\to B/Z(B)$ is
surjective.
Then the following conditions are equivalent.
\begin{equation}\begin{minipage}[c]{35pc}\label{d.gp_no_Us}
$\pi f:\prod_{i\in I} G_i\to B/Z(B)$
does {\em not} factor through the natural map $\prod_{i\in I} G_i\to
\prod_{i\in I} G_i/\U_0\times\dots\times\prod_{i\in I} G_i/\U_{n-1}$
for any finite family of ultrafilters $\U_0,\dots,\U_{n-1}$ on $I.$
\end{minipage}\end{equation}
\begin{equation}\begin{minipage}[c]{35pc}\label{d.gp_Jn}
There exists a partition of $I$ into countably many
subsets $J_0,\,J_1,\dots\,,$ such that each subgroup
$\prod_{i\in J_n} G_i\subseteq \prod_{i\in I} G_i$ contains a
pair of elements $x_n,$ $y_n$
whose images in $B$ under $f$ do not commute.
\end{minipage}\end{equation}
\end{lemma}

\begin{proof}
The easy direction is \eqref{d.gp_Jn}$\!\implies\!$\eqref{d.gp_no_Us}.
The fact that $f(x_n)$ and $f(y_n)$ do not commute
tells us, in particular, that $f(x_n)\notin Z(B).$
Hence for $\F$ defined by~\eqref{d.F} (noting in particular the
last line thereof), $I-J_n\notin\F.$
Since this is true for each $n,$~\eqref{d.cap_ultra}
tells us that the filter $\F$ is not a finite intersection of
ultrafilters, giving~\eqref{d.gp_no_Us}.

To get the converse, note that if~\eqref{d.gp_no_Us} holds,
equivalently, if $\F$ is not a finite intersection of
ultrafilters, then by~\eqref{d.cap_ultra}
we can partition $I$ into subsets $J_0,\,J_1,\dots,$
none of whose complements lies in $\F;$ i.e., by the
last line of~\eqref{d.F}, such
that each $\prod_{i\in J_n} G_i$ contains an element $x_n$ which
is mapped by $f$ to a noncentral element of $B.$
Fixing $n,$ this says that there exists an element $b\in B$
which does not commute with $f(x_n).$
I claim we can take such a $b$ to be the image of an element
$y\in\prod_{i\in I} G_i$ under $f.$
Indeed, if $f$ is surjective, this is immediate.
If instead we have the weaker hypothesis that
$\pi f:\prod_{i\in I} G_i\to B\to B/Z(B)$ is surjective, then we
can choose $y\in\prod_{i\in I} G_i$ whose image under $f$
is congruent to $b$ modulo $Z(B).$
Since multiplication by an element of $Z(B)$ does not affect
what members of $B$ an element commutes with,
$f(y)$ does not commute with~$f(x_n).$

Let us now write $y=y_n y',$ where $y_n\in\prod_{i\in J_n} G_i$
while $y'\in\prod_{i\in I-J_n} G_i.$
Then $y'$ commutes with $x_n,$ since they have disjoint
supports in our product group.
Hence $f(y')$ commutes with $f(x_n);$
hence if $f(y_n)$ also commuted with $f(x_n),$
then $f(y)=f(y_n)\,f(y')$ would commute
with $f(x_n),$ contradicting our choice of $y.$
Hence, rather, $x_n,\,y_n\in\prod_{i\in J_n} G_i$ have images in $B$
which do not commute, giving~\eqref{d.gp_Jn}.
\end{proof}

We can now get the first of our results showing that
any group $B$ admitting a map $f$ satisfying~\eqref{d.gp_no_Us}
must be ``big''.

\begin{theorem}\label{T.gp_main}
Let $B$ be a group such that there exist a family of groups
$(G_i)_{i\in I},$ and a group homomorphism
$f:\prod_{i\in I} G_i\to B,$ for which the induced homomorphism
$\pi f:\prod_{i\in I} G_i\to B/Z(B)$ does not factor
through the projection of $\prod_{i\in I} G_i$ to the product
of finitely many ultraproducts of the $G_i.$
Then $B$ contains families of elements
$(a_S)_{S\subseteq\omega},$ $(b_S)_{S\subseteq\omega},$
indexed by the subsets $S$ of $\omega,$ such that
\begin{equation}\begin{minipage}[c]{35pc}\label{d.commute}
All the elements $a_S$ $(S\subseteq\omega)$ commute with one another,
and all the elements $b_S$ $(S\subseteq\omega)$
likewise commute with one another.
\end{minipage}\end{equation}
\begin{equation}\begin{minipage}[c]{35pc}\label{d.a_disj}
For $S$ and $T$ {\em disjoint} subsets of $\omega,$
one has $a_S\,a_T=a_{S\cup T},$ $b_S\,b_T=b_{S\cup T},$
and $a_S\,b_T=b_T\,a_S.$
\end{minipage}\end{equation}
\begin{equation}\begin{minipage}[c]{35pc}\label{d.one}
For subsets $S$ and $T$ of $\omega$ with $\card(S\cap T)=1,$
$a_S\,b_T\neq b_T\,a_S.$
\end{minipage}\end{equation}
\end{theorem}

\begin{proof}
Given $G_i$ and $f$ as in the hypothesis,
i.e., satisfying~\eqref{d.gp_no_Us}, Lemma~\ref{L.gp_Us}
gives us sets $J_n\subseteq I$ and elements
$x_n,$ $y_n$ $(n\in\omega)$ as in~\eqref{d.gp_Jn}.
Let $H_n=\prod_{i\in J_n} G_i\subseteq\prod_{i\in I} G_i$
$(n\in\omega),$
so that we can regard $\prod_{i\in I} G_i$ as $\prod_{n\in\omega}H_n,$
the $x_n$ and $y_n$ as elements of that group with
singleton supports, and $f$ as a homomorphism
$\prod_{n\in\omega}H_n\to B.$

For each subset $S\subseteq\omega,$ let $x_S$ be the
element of $\prod_{n\in\omega}H_n$ whose component at $n$
is $x_n$ if $n\in S,$ and $e$ otherwise, and let elements
$y_S$ be obtained similarly from the $y_n.$
It is easy to see that any two elements $x_S$ and $x_T$
commute with one another in $\prod_{n\in\omega}H_n,$
and similarly for the $\!y\!$'s;
and that for $S$ and $T$ disjoint,
$x_S\,x_T=x_{S\cup T},$ $y_S\,y_T=y_{S\cup T},$ and
$x_S\,y_T=y_T\,x_S.$
Hence, letting $a_S=f(x_S),$ $b_S=f(y_S),$ we get~\eqref{d.commute}
and~\eqref{d.a_disj}.

For general $S$ and $T,$ the commutator
$[x_S,y_T]$ will have $\!n\!$-th component
$[x_n,y_n]$ if $n\in S\cap T,$ and $e$ otherwise.
So if $S\cap T$ is exactly $\{n\}$ for some $n\in\omega,$
then $f([x_S,y_T])=f([x_n,y_n]),$ which by choice
of $x_n$ and $y_n$ is not $e,$ giving~\eqref{d.one}.
\end{proof}

By restricting the elements $b_T$ that we consider, we can get a
clearer view of the behavior of the elements $a_S$:

\begin{corollary}\label{C.c}
In the situation of Theorem~\ref{T.gp_main}, an element
$a_S$ $(S\subseteq\omega)$ commutes with an element $b_{\{n\}}$
$(n\in\omega)$ if and only if $n\notin S.$
Thus, the elements $a_S$ exhibit all possible combinations
of which members of the countable set $\{b_{\{n\}}\mid n\in\omega\}$
they commute with.
Hence they are distinct modulo $Z(B);$ so their
images in $B/Z(B)$ generate a commutative
subgroup of continuum cardinality.
\end{corollary}

\begin{proof}
The first sentence is immediate from~\eqref{d.a_disj}
and~\eqref{d.one}, and clearly implies the second.
Since multiplication by a member of $Z(B)$ does not affect
what elements a member of $B$ commutes with, elements
which can be distinguished by the latter properties
are necessarily distinct modulo $Z(B).$
The group generated by the $a_S$ is commutative in view
of~\eqref{d.commute}, hence so is the image
of that group in $B/Z(B).$
\end{proof}

Above we have obtained ``element-theoretic'' consequences of
the existence of a map $\prod_{i\in I} A_i\to B$ that does not
factor through finitely many ultrafilters.
There are also ``subgroup-theoretic'' consequences.
We shall find it convenient to state some of these, not in terms
of image subgroups $\pi f(\prod_{i\in S} G_i)\subseteq B/Z(B),$
but in terms of the inverse images
$f(\prod_{i\in S} G_i)Z(B)$ of those subgroups in $B.$
Let us start by noting some general properties of this construction,
independent of whether $\pi f$ factors through
finitely many ultraproducts.

\begin{lemma}\label{L.B_S}
Let $B$ be a group, $(G_i)_{i\in I}$ a family of groups,
and $f:\prod_{i\in I} G_i\to B$ a homomorphism
which is surjective \textup{(}or more generally, satisfies
$B=f(\prod_{i\in I} G_i)\,Z(B)).$
For every subset $S\subseteq I,$ let
\begin{equation}\begin{minipage}[c]{35pc}\label{d.B_S}
$B_S=f(\prod_{i\in S}G_i)\,Z(B),$ a normal subgroup of $B.$
\end{minipage}\end{equation}
Then
\begin{equation}\begin{minipage}[c]{35pc}\label{d.capcup}
$B_\emptyset=Z(B),$ $B_I=B,$ and
for $S,\,T\subseteq I,$ one has $B_S\,B_T=B_{S\cup T}$ and
$B_S\cap B_T=B_{S\cap T}.$
\end{minipage}\end{equation}
\begin{equation}\begin{minipage}[c]{35pc}\label{d.centralizer}
For $S,\,T\subseteq I,$
the centralizer of $B_T$ in $B_S$ is $B_{S-T}.$
\end{minipage}\end{equation}
Hence \textup{(}again writing $\pi:B\to B/Z(B)$ for the
quotient map\textup{)},
\begin{equation}\begin{minipage}[c]{35pc}\label{d.prod}
For {\em disjoint} subsets $S,\,T\subseteq I,$
$\pi(B_{S\cup T})$ is the direct product of its subgroups
$\pi(B_S)$ and $\pi(B_T).$
\end{minipage}\end{equation}
Moreover,
\begin{equation}\begin{minipage}[c]{35pc}\label{d.surj}
If $(S_k)_{k\in K}$ is a family of pairwise disjoint subsets of $I,$
and we let $S=\bigcup_{k\in K} S_k,$ then the
map $\pi(B_S)\to\prod_{k\in K}\pi(B_{S_k})$ determined by the
projections $\pi(B_S)\to\pi(B_{S_k})$ \textup{(}which
by~\eqref{d.prod} is an isomorphism if $K$ is finite\textup{)}
is always surjective.
\end{minipage}\end{equation}
\end{lemma}

\begin{proof}
That each $B_S$ is normal in $B,$ as asserted in~\eqref{d.B_S},
follows from the normality
of $\prod_{i\in S}G_i$ in $\prod_{i\in I} G_i,$ and
the centrality of $Z(B)$ in $B.$

The first three equalities
of~\eqref{d.capcup} are immediate, as is the direction
$B_S\cap B_T\supseteq B_{S\cap T}$ of the final equality.
Before proving the reverse inclusion,
let us note a case of~\eqref{d.centralizer} which is also immediate:
\begin{equation}\begin{minipage}[c]{35pc}\label{d.centralize}
If $S$ and $T$ are disjoint subsets of $I,$
then $B_S$ and $B_T$ centralize one another.
\end{minipage}\end{equation}

To get the remaining part of~\eqref{d.capcup},
$B_S\cap B_T\subseteq B_{S\cap T},$ consider an element of the
left-hand side, which we may write
\begin{equation}\begin{minipage}[c]{35pc}\label{d.u&v}
$f(u)\,z_1=f(v)\,z_2,$ where $u\in\prod_{i\in S}G_i,$
$v\in\prod_{i\in T}G_i,$ and $z_1,\,z_2\in Z(B).$
\end{minipage}\end{equation}
Let us write $u=u' u'',$ where $u'\in\prod_{i\in S\cap T}G_i$
and $u''\in\prod_{i\in S-T} G_i.$
Thus our element~\eqref{d.u&v} becomes $f(u')\,f(u'')\,z_1.$
Since $u'\in\prod_{i\in S\cap T}G_i,$
if we can show that $f(u'')\in Z(B),$ then~\eqref{d.u&v} will lie in
$f(\prod_{i\in S\cap T} G_i)\,Z(B)=B_{S\cap T},$ as required.

Thus, we need to show that $f(u'')$ centralizes
$B=B_{I}=B_{S-T}\,B_{I-(S-T)}.$
Since $f(u'')\in B_{S-T},$ it certainly centralizes $B_{I-(S-T)}.$
On the other hand if we write the equation in~\eqref{d.u&v} as
$f(u')\,f(u'')\,z_1=f(v)\,z_2,$ equivalently,
$f(u'')=f(u')^{-1}f(v)\,z_2\,z_1^{-1},$ we see that
all the factors on the right lie in $B_T,$
hence centralize $B_{S-T}.$
Hence so does $f(u''),$ completing the proof
of the last assertion of~\eqref{d.capcup}.

We can now easily prove~\eqref{d.centralizer}.
By~\eqref{d.capcup}, $B_{S-T}$ is contained in $B_S,$ and
by~\eqref{d.centralize}, it centralizes $B_T,$ so we need
only show that conversely, any element of $B_S$ that
centralizes $B_T$ lies in $B_{S-T}.$
As in the preceding argument, we can write our element of $B_S$ as
$f(u')\,f(u'')\,z,$ where $u'\in\prod_{i\in S\cap T}G_i,$
$u''\in\prod_{i\in S-T}G_i.$
This time, we need to prove that $f(u')\in Z(B).$
Now since $f(u')\,f(u'')\,z$ centralizes $B_T,$
and $f(u'')$ and $z$ automatically do, we see that $f(u')$
centralizes $B_T.$
Also, since $f(u')\in B_{S\cap T},$ and $S\cap T$ is disjoint
from $I-T,$ $f(u')$ centralizes $B_{I-T}.$
Hence it centralizes $B_T\,B_{I-T}=B,$ so it lies in $Z(B),$ as claimed.

\eqref{d.prod} follows easily
from~\eqref{d.centralizer} and~\eqref{d.capcup}.

To establish~\eqref{d.surj}, we take an element
of $\prod_{k\in K} \pi(B_{S_k}),$ lift its component
in each $\pi(B_{S_k})$ to an element of $\prod_{i\in S_k} G_i,$
regard these together as giving an element of $\prod_{i\in S} G_i,$
and note that the image of this
element in $\pi(B_S)$ has the desired property.
\end{proof}

Note that in the situation of the above lemma, the subgroups $B_S$
need not be distinct for distinct $S\subseteq I.$
For instance, if we take a family $(G_i)_{i\in I}$ of noncommutative
groups and an ultrafilter $\U$ on $I,$ let
$B=\prod_{i\in I} G_i/\U,$
and let $f:\prod_{i\in I} G_i\to B$ be the quotient map,
then the above construction gives only two distinct subgroups of $B:$
$B_S=B$ if $S\in\U,$ and $B_S=Z(B)$ otherwise.

We shall now get a factorization-through-ultraproducts result from
the above lemma.
Let us (following \cite[\S4.3]{cap_ultra}) call
subgroups $B',$ $B''$ of a group $B$ {\em almost direct factors}
if $B=B'B'',$ and each of $B',$ $B''$
is the centralizer in $B$ of the other.
A subgroup $B'\subseteq B$
belonging to such a pair (equivalently, such that $B'$
is its own double centralizer in $B,$ and $B$ is the product of
$B'$ and its centralizer) will thus be called
an {\em almost direct factor of $B.$}
We shall say $B$ has {\em chain condition on almost direct factors}
if the partially ordered set of almost direct factors of $B$
has ascending chain condition, equivalently
(since that partially ordered
set is self-dual under the operation of taking centralizers),
if it has descending chain condition.
(As noted in~\cite{cap_ultra}, these are the analogs for
groups of definitions first made for algebras in~\cite[\S6]{prod_Lie1}.)

Observe that in the situation treated in Lemma~\ref{L.B_S},
statements~\eqref{d.centralizer} and~\eqref{d.capcup}
show that for every $S\subseteq I,$ the subgroups $B_S,$ $B_{I-S}$ are
a pair of almost direct factors of $B.$
We deduce

\begin{theorem}[{$=$\cite[Proposition~10]{cap_ultra}}]\label{T.gp_CC}
Let $B$ be a group, and suppose that there exist a family of groups
$(G_i)_{i\in I},$ and a homomorphism
$f:\prod_{i\in I} G_i\to B,$ such that the induced homomorphism
$\pi f:\prod_{i\in I} G_i\to B/Z(B)$ is surjective, and does not factor
through the natural projection of $\prod_{i\in I} G_i$ to any finite
product of ultraproducts of the $G_i.$

Then $B$ does not have chain condition on almost direct factors.
In fact, it has a family of almost direct factors order-isomorphic
to the lattice $2^\omega,$ and forming a sublattice
of the lattice of subgroups of $B.$
\end{theorem}

\begin{proof}
Given $(G_i)_{i\in I}$ with the indicated non-factorization
property, let $J_0,\,J_1,\dots$ be as in Lemma~\ref{L.gp_Us}.
To every subset $S$ of $\omega,$ let us associate
the subgroup $B_{\bigcup_{n\in S} J_n}.$
From Lemma~\ref{L.B_S} we see that each of these subgroups is an
almost direct factor of $B,$ and that the lattice relations among
the subsets of $\omega$ are also satisfied by the corresponding
subgroups; so it will suffice to show that {\em non-inclusions}
of subsets of $\omega$ yield non-inclusions of subgroups.
If $S\not\subseteq T,$ take $m\in S-T.$
By our assumption on the $J_n,$ the subgroup $B_{J_m}$ is
not self-centralizing, hence though it centralizes
$B_{\bigcup_{n\in T} J_n},$ it does not centralize
$B_{\bigcup_{n\in S} J_n};$
so the latter is not contained in the former.
\end{proof}

Neither of the conclusions of
Theorem~\ref{T.gp_main} and Theorem~\ref{T.gp_CC} implies the other.
To get examples of these non-implications, let $G$ be a simple group.

If we embed $G^\omega$ in any simple overgroup $B,$ then $B$
inherits from $G^\omega$ families of elements $a_S,$ $b_S$ as
in Theorem~\ref{T.gp_main}; but being simple, $B$
has no nontrivial almost direct decompositions,
hence it satisfies chain condition on almost direct factors,
i.e., fails to satisfy the conclusion of Theorem~\ref{T.gp_CC}.

On the other hand, if we take for $B$ the group
$\bigoplus_\omega G$ of elements of $G^\omega$ having finite support,
and let $B_S=\bigoplus_S G$
for each $S\subseteq\omega,$ we find that these
subgroups satisfy~\eqref{d.capcup}-\eqref{d.prod}, hence
constitute a system of almost direct factors lattice-isomorphic
to $2^\omega,$ as in Theorem~\ref{T.gp_CC}.
But if $G$ is countable, $B$ will also be so, so
it cannot satisfy the conclusion of Theorem~\ref{T.gp_main}.

So neither of these groups $B$
admits a surjective homomorphism $f$ from a direct product group
such that $\pi f$ (which in both cases would
be $f,$ since $Z(B)$ is trivial)
fails to factor through finitely many ultraproducts.
However, in the first case, only
Theorem~\ref{T.gp_CC} rules this out,
while in the second, only Theorem~\ref{T.gp_main} does.

Though the above example with $B=\bigoplus_\omega G$
satisfies~\eqref{d.capcup}-\eqref{d.prod},
it does not satisfy~\eqref{d.surj}, as can be seen by taking for
the $S_k$ the singleton subsets of $\omega.$
One may ask whether for any group $B,$ every system of subgroups
$B_S$ $(S\subseteq I)$ of $B$ that
satisfies all of~\eqref{d.capcup}-\eqref{d.surj} arises as
in Lemma~\ref{L.B_S}.

The answer is still negative.
For instance, suppose $B$ is a group which has trivial center,
and which cannot be written as a homomorphic image
of a nonprincipal ultraproduct of
a family of groups indexed by $\omega.$
(We shall see in \S\ref{S.gp_prin} that the free group on two
generators, among many others, cannot be so written.)
Suppose we take a nonprincipal ultrafilter $\U$
on $\omega,$ and define $B_S\subseteq B$ to be all
of $B$ whenever $S\in\U,$ and $\{e\}$ otherwise.
It is not hard to verify that this family
satisfies~\eqref{d.capcup}-\eqref{d.surj}, but that if
it arose as in Lemma~\ref{L.B_S} (with $\omega$
for $I),$ then $B$ would be
a homomorphic image of $\prod_{n\in\omega} G_n/\U,$ contradicting our
choice of $B.$

We record a special case of Theorem~\ref{T.gp_CC}, for
easy application to some later examples.

\begin{corollary}[to Theorem~\ref{T.gp_CC} and its proof]\label{C.one_U}
Suppose $B$ is a group with trivial center, and having no
nontrivial direct product decomposition.
Then every homomorphism from a direct product group
$\prod_{i\in I} G_i$ onto $B$ factors through a
single ultraproduct $\prod_{i\in I} G_i/\U$ of the $G_i.$\qed
\end{corollary}

\section{Further examples}\label{S.gp_eg}

Theorem~\ref{T.gp_CC} shows that a group $B$ which admits
a surjective homomorphism from an infinite direct product group that
does not factor through finitely many ultraproducts looks, itself, in
some ways, like an infinite direct product -- at least after we
divide out $Z(B).$
The next example shows that this behavior of $B/Z(B)$
can coexist with very un-product-like behavior in $Z(B).$

\begin{example}\label{E.Heisenberg}
Groups $B$ and $G$ and a homomorphism $f:G^\omega\to B$
such that the induced subgroups $B_S$ $(S\subseteq\omega)$
are all distinct, but such that the center of
each of the given copies of $G$ in $G^\omega$
is mapped isomorphically to $Z(B)\neq\{e\};$
and which also show that in~\eqref{d.one}, the hypothesis
$\card(S\cap T)=1$ cannot be weakened to merely say that
$S\cap T$ is nonempty and finite.
\end{example}

\begin{proof}[Construction and proof]
Let $k$ be a field, and $G$ the Heisenberg group over $k;$
that is, the multiplicative group of upper triangular $3\times 3$
matrices with $\!1\!$'s on the main diagonal; equivalently,
the group of $\!3\!$-tuples of elements of $k$ under the multiplication
$(a,a',a'')\,(b,b',b'')=(a{+}b,\,a'{+}b',\,a''{+}b''{+}a'b).$
Clearly, the countable
power group $G^\omega$ can be described as the group of
$\!3\!$-tuples of elements of the power
ring $k^\omega$ under the operation given by the same formula.

Let us now take the $\!k\!$-vector-space homomorphism
$s:\bigoplus_\omega k\to k$ which for each $n$
acts on the $\!n\!$-th direct summand by $1\mapsto s_n,$
for some specified elements $s_n\in k-\{0\},$ and by linear algebra,
let us extend $s$ to a vector-space
homomorphism $\sigma: k^\omega\to k.$
Let $B$ be the homomorphic image of $G^\omega$ gotten by
dividing $Z(G^\omega)=k^\omega$ by $\r{ker}(\sigma).$
This can be described as
\begin{equation}\begin{minipage}[c]{35pc}\label{d.kwkwk}
$k^\omega\times k^\omega\times k,$ under the operation
$(a,a',a'')\,(b,b',b'')=(a+b,\,a'+b',\,a''+b''+\sigma(a' b)).$
\end{minipage}\end{equation}

I claim that $Z(B)=\{0\}\times\{0\}\times k.$
To see this, let us first show that every
$(a,b,c)\in B$ with $a\neq 0$ is noncentral.
Choose $n$ such that $a$ has $\!n\!$-th component $a_n\neq 0,$
and take $b'\in k^\omega$ to have $1$ in the
$\!n\!$-th position and $0$ in all others.
Then we find that
the commutator of $(a,b,c)$ and $(0,b',0)$ is $(0,0,s_n\,a_n)\neq e.$
The analogous argument shows $(a,b,c)$ noncentral if $b\neq 0.$
Elements $(0,0,c)$ are clearly central, so we get the asserted
description of $Z(B),$ and we see that this is the image in $B$ of
the center of each of our copies of $G$ in $G^\omega,$
and indeed, of the center of $G^S\subseteq G^\omega$
whenever $\emptyset\neq S\subseteq\omega.$

So though the images in $B/Z(B)$ of these subgroups $G^S$
are the corresponding factors
$(k\times k)^S\subseteq (k\times k)^\omega,$
when we look at the images in $Z(B)$ of their centers,
the distinctions among them disappear.

To get the final assertion of this example, let us
partition $\omega$ into the singletons $J_n=\{n\},$
so that in the notation of the proof of Theorems~\ref{T.gp_main},
each $H_n$ is $G.$
For each $n,$ let $x_n=(1,0,0),\ y_n=(0,1,0)$ in $H_n,$
and let us use these to construct elements
$a_S,\,b_T\in B$ as in that proof.
Then if $S$ and $T$ are subsets of $\omega$ which
intersect in a finite set $\{n_0,\dots,n_{d-1}\},$ we see that
in $B$ the commutator $[a_S,\,b_T]$
is $(0,0,s_{n_0}+\dots+s_{n_{d-1}}).$
If $d=1$ this is necessarily a nonidentity element, as stated
in~\eqref{d.one}; but if $S$ and $T$ intersect in more than
one element, this may or may not be true, depending on the
choice of the $s_n.$
(In particular, if the field $k$ is finite, then whatever the $s_n,$
there must be some nonempty family of $\leq\r{card}(k)$ $\!s_n\!$'s
that sum to zero.)
So the restriction $\card(S\cap T)=1$
in~\eqref{d.one} cannot be dropped.
\end{proof}

In the above example, the focus was on the part of the
map going into $Z(B);$ the map $G^\omega\to B/Z(B)$ was
a straightforward homomorphism of direct products.
But this is not always the case; that is,
the maps which~\eqref{d.surj} shows to be
surjective need not, in general, be isomorphisms.
For instance, in the example mentioned immediately after the
proof of Lemma~\ref{L.B_S}, where
the $G_i$ were arbitrary noncommutative groups,
and $f$ was the map $\prod_{i\in I} G_i\to(\prod_{i\in I} G_i)/\U,$
for $\U$ an ultrafilter on $I,$ if $\U$ is nonprincipal and we take for
the $S_k$ all the singletons $\{i\}$ $(i\in I),$ so that
$S=I,$ then each $\pi(B_{S_k})$ is trivial, but $\pi(B_S)$ is not.

One can, of course, modify this example to get one which also has the
property that every $G_i$ has nontrivial image in $B/Z(B):$

\begin{example}\label{E.prodXultra}
A group homomorphism $\prod_{n\in\omega}G_n\to B$ where
all the $B_{\{n\}}/Z(B)$ are nonzero \textup{(}so that all
the $B_S$ are distinct\textup{)}, but not all the
surjections of~\eqref{d.surj} are isomorphisms.
\end{example}

\begin{proof}[Construction]
Let $G_n$ $(n\in\omega)$ be groups with trivial centers,
each having a proper nontrivial normal subgroup $N_n\lhd G_n$
such that $G_n/N_n$ also has trivial center.
Let $\U$ be any nonprincipal ultrafilter on $\omega,$
let $H=(\prod_{n\in\omega} G_n)/\U,$ and
let $f: \prod_{n\in\omega} G_n\to H\times\prod_{n\in\omega} G_n/N_n$
be the map obtained from the obvious homomorphisms
$\prod_{n\in\omega} G_n\to H$ and
$\prod_{n\in\omega} G_n\to \prod_{n\in\omega} G_n/N_n.$
Let $B$ be the image of~$f.$

For $S\subseteq\omega,$ what does $B_S$ look like?
This depends on whether or not $S\in\U.$
If not, we see that $B_S=\prod_{n\in S} G_n/N_n;$
in particular, for every $n\in\omega$ we have $B_{\{n\}}=G_n/N_n.$
Thus for $S\notin\U,$ the group $B_S$ can be identified
with $\prod_{n\in S} B_{\{n\}}.$
However, when $S\in\U,$ the $\!H\!$-component of $B_S$ will be
the full group $H,$ which carries structure
from the normal subgroups $N_n$ which is ignored by each
group $B_{\{n\}};$ so in these cases, the natural
map $B_S\to\prod_{n\in S} B_{\{n\}}$ is not one-to-one.
\end{proof}

One can generalize the above construction by
replacing $\U$ with an arbitrary filter $\F,$
though the description of the groups $B_S$ is more complicated
to state when $S$ is neither a member of $\F$ nor the
complement of one.
And, of course, one can set up examples based on
more than one system of normal subgroups and more than one filter.

In the above example, though the system of
subgroups $B_S$ described does not have the property that the
maps of~\eqref{d.surj} are isomorphisms, the
group $B$ has other systems of subgroups that can be
shown to have that property.
Here is an example having no such family.

\begin{example}\label{E.finitely_many}
A group $B$ having elements $a_S$ and $b_S$ $(S\subseteq\omega)$
satisfying~\eqref{d.commute}-\eqref{d.one},
and distinct subgroups $B_S$ $(S\subseteq\omega)$
satisfying~\eqref{d.B_S}-\eqref{d.prod},
but having no such system of distinct
subgroups also satisfying~\eqref{d.surj} for any infinite
family of disjoint nonempty sets $(s_k)_{k\in K};$ so that $B$
cannot admit a surjective homomorphism from a direct
product group which does not factor through the
product of finitely many ultraproducts.
\end{example}

\begin{proof}[Construction and sketch of proof]
Let $G$ be an infinite simple group, and $B$ the subgroup of
$G^\omega$ consisting of those $\!\omega\!$-tuples
assuming only finitely many distinct values in $G.$
If we choose a pair
of noncommuting elements $x,y\in G,$ and for each $n\in\omega$
let $x_n$ be the element $x$ of the $\!n\!$-th copy of $G,$
and $y_n$ the element $y$ thereof, then we see that the
elements $x_S$ and $y_S$ $(S\subseteq\omega),$
constructed as in Theorem~\ref{T.gp_main},
will lie in $B,$ and, renamed $a_S$ and $b_S,$ will
satisfy~\eqref{d.commute}-\eqref{d.one}.
Similarly, if we let $B_S$ be the subgroup
of $B$ consisting of elements with support in $S,$
then~\eqref{d.capcup}-\eqref{d.prod} are immediate.

I will now sketch why $B$ admits no system of nontrivial almost direct
factors $B_{S_k}$ and $B_S$
satisfying~\eqref{d.surj} for any infinite $K.$
Note that $B$ has trivial center, so that almost direct
factors are simply direct factors.
Now it is easy to
verify using the simplicity of $G$ that if $B$ has a direct
product decomposition $B=B'\times B'',$ then
for each $n\in\omega,$ one of $B',$ $B''$ has as $\!n\!$-th coordinates
all members of $G,$ while the other has only $e$
in that coordinate.
From this one can deduce that every such decomposition has the form
$B'=B_S,$ $B''=B_{\omega-S}$ for some $S\subseteq\omega.$
We can now combine the ``finitely many distinct values''
condition in the definition of $B$ with the fact that $G$
is infinite to see the impossibility of an
infinite family of nontrivial almost direct factors $B_{S_k}$
satisfying the surjectivity condition~\eqref{d.surj}.

Lemma~\ref{L.B_S} and
the method of proof of Theorem~\ref{T.gp_CC} now show that
every homomorphism from a direct product onto $B$ must
factor through finitely many ultraproducts.
\end{proof}

(For some other results on the subgroup of a power
group $G^I$ consisting of
the elements with only finitely many distinct coordinates --
though for abelian groups -- see~\cite{Z[B]}.)

\section{Conditions forcing the ultrafilters to be principal}\label{S.gp_prin}

We have obtained conditions that force group homomorphisms
$\prod_{i\in I} G_i\to B$ to factor through the direct product
of finitely many ultraproducts of the $G_i.$
When can we say that any map that so factors must in fact factor
through the product of finitely many $G_i;$ i.e., that the ultrafilters
involved must be principal?

Here set-theoretic considerations come in.
If $\kappa$ is a {\em measurable cardinal}, then sets $I$ of
cardinality $\geq\kappa$ admit nonprincipal {\em $\!\kappa\!$-complete}
ultrafilters; that is, ultrafilters closed under
all $\!{<}\kappa\!$-fold intersections.
(Two quick terminological notes: (i)~The condition of being closed under
countable intersections, which by the above definition is
$\!\aleph_1\!$-completeness, is also called {\em countable
completeness}.
(ii)~We shall follow the definition of measurable cardinal
used in \cite{Ch+Keis}, which counts $\aleph_0$ as measurable; so we
will write ``uncountable measurable cardinal''
for what many authors simply call a measurable cardinal.)

If $\kappa$ is an uncountable measurable cardinal and $I$
a set of cardinality $\geq\kappa,$ and we take a family
$(G_i)_{i\in I}$ of groups (or more generally, of any sort of
algebraic structures defined by finitely many
finitary operations) whose cardinalities have
a common bound ${<}\kappa,$ then their ultraproducts
with respect to $\!\kappa\!$-complete ultrafilters behave very
much as do ordinary ultraproducts of finite groups with a
common finite bound on their orders;
to wit, every such ultraproduct is isomorphic to one of the $G_i.$
Hence, if there exists such a cardinal $\kappa,$ then
{\em every} group $B$ of cardinality ${<}\kappa$ can be represented
as an ultrapower of itself with respect to a nonprincipal
$\!\kappa\!$-complete ultrafilter $\U.$
So for every such $B$
we get a surjective homomorphism $B^I\to B$ which factors through
the ultrapower $B^I/\U$ but not through finitely many projection maps --
which seems to be bad news for the type of result we are hoping for.

However, it is known that if uncountable measurable cardinals
exist, they must be quite enormous \cite[Theorem~4.2.14]{Ch+Keis},
and that if the standard set theory,
ZFC, is consistent, it is consistent with the nonexistence
of such cardinals.
Hence it would be reasonable to work under the assumption that
no uncountable measurable cardinals exist, or, if they exist, to
restrict our index-sets to cardinalities less than all such cardinals.

The next observation shows that when doing the
spade-work of our investigation, we can in fact
restrict attention to the case where our index-set is countable.

\begin{lemma}\label{L.countable}
If $B$ is a group, then the following conditions are equivalent:
\begin{equation}\begin{minipage}[c]{35pc}\label{d.non-meas}
$B$ is a homomorphic image of an ultraproduct of a family of groups
indexed by an arbitrary set $I,$ with respect to some ultrafilter
$\U$ on $I$ that is not countably complete,
equivalently, that is not $\!\kappa\!$-complete
for any uncountable measurable cardinal~$\kappa.$
\end{minipage}\end{equation}
\begin{equation}\begin{minipage}[c]{35pc}\label{d.omega}
$B$ is a homomorphic image of an ultraproduct of a family of groups
indexed by $\omega,$ with
respect to a nonprincipal ultrafilter on $\omega.$
\end{minipage}\end{equation}

The same is true with ``groups'' replaced by objects
of any other variety of finitary algebras, in the sense
of universal algebra.
\end{lemma}

\begin{proof}
The equivalence referred to in~\eqref{d.non-meas} follows from
the fact that any countably complete
ultrafilter must be $\!\kappa\!$-complete
for some uncountable measurable cardinal $\kappa$
\cite[Proposition 4.2.7]{Ch+Keis}.

Since a nonprincipal ultrafilter on $\omega$ is not countably complete,
we have~\eqref{d.omega}$\!\implies\!$\eqref{d.non-meas}.
On the other hand, it is easy to show that if $\U$ is a
non-countably-complete
ultrafilter on a set $I,$ then $I$ can be partitioned
as $\bigcup_{n\in\omega} J_n$ where no $J_n$ belongs to $\U.$
In this situation we find
that $\{S\subseteq\omega\mid \bigcup_{n\in S} J_n\in\U\}$
is a nonprincipal ultrafilter $\U'$ on $\omega,$
and that given groups $G_i$ $(i\in I),$ the natural map
$\prod_{i\in I} G_i\to \prod_{i\in I} G_i/\U$ factors through
$\prod_{n\in\omega}(\prod_{i\in J_n} G_i)/\U'.$
Hence, writing $\prod_{i\in J_n} G_i=H_n,$ we see that
if, as in~\eqref{d.non-meas}, $B$ is a homomorphic image
of $\prod_{i\in I} G_i/\U,$ then
it is also a homomorphic image of $\prod_{n\in\omega} H_n/\U',$
giving~\eqref{d.omega}.

The final assertion is clear.
(The assumption that our algebras are finitary is needed to insure that
algebra structures are induced on ultraproducts of such algebras.)
\end{proof}

So below, it will suffice to examine which groups
are homomorphic images of nonprincipal ultraproducts of
countable families of groups.
For brevity, we shall call an ultraproduct
of a countable family a ``countable ultraproduct''.

My first guess was that if $B$ was such a homomorphic image, then
the cardinality of $B/Z(B)$ would have to be either finite
or at least the cardinality of the continuum.
But Tom Scanlon suggested the following counterexample.

\begin{lemma}[T.\,Scanlon, personal communication]\label{L.TS}
Let $B$ be the semidirect product of the additive group $\Q$
of rational numbers, and the $\!2\!$-element group $\{\pm 1\},$
determined by the multiplicative action of the latter on the former.
\textup{(}I.e., $B$ has underlying set $\{\pm 1\}\times\Q,$ and
multiplication $(\alpha,a)(\beta,b)=(\alpha\beta,\beta a + b).)$

Then every ultrapower of $B$ admits a homomorphism onto $B.$
Hence though $B=B/Z(B)$ is countable, it is a homomorphic image
of a nonprincipal countable ultraproduct of groups.
\end{lemma}

\begin{proof}
Clearly, the only elements of $B$ that commute with $(-1,0)$
are those with second component $0,$ while the only elements
that commute with $(1,1)$ are those with first component $1;$
so $Z(B)=\{e\},$ justifying the formula $B=B/Z(B).$

It is easy to see that for any ultrafilter $\U$ on any index set $I,$
the ultrapower $B^I/\U$ will be the semidirect product
of $\{\pm 1\}$ and $\Q^I/\U$
determined by the natural action of the former group on the latter.
Now $\Q^I/\U,$ like $\Q,$
is a nontrivial torsion-free divisible group, i.e.,
a nontrivial $\!\Q\!$-vector-space, and, as such, admits a surjective
homomorphism $\varphi:\Q^I/\U\to\Q.$
The map $B^I/\U\to B$ given by
$(\alpha,\beta)\mapsto(\alpha,\varphi(\beta))$ is easily seen
to be a surjective homomorphism, as claimed.
\end{proof}

By Corollary~\ref{C.one_U}, every homomorphism from a direct product
group $\prod_{i\in I} G_i$ onto the above group $B$
factors through a single ultraproduct of the $G_i;$
but the above result shows that (even when the index-set is
countable) the ultrafilter involved need not be principal.

In fact, the only condition I know that guarantees
factorization through finitely many of the $G_i$ is based on
requiring appropriate abelian subgroups of $B$ to satisfy
similar factorization properties as abelian groups.
The key observation is

\begin{lemma}\label{L.ultra}
Suppose $B$ is a homomorphic image of a nonprincipal
countable ultraproduct of groups, $(\prod_{n\in\omega} G_n)/\U.$
Then every element $b\in B$ lies in a homomorphic image within
$B$ of $\Z^\omega/\U,$ a nonprincipal countable ultrapower of $\Z.$
\end{lemma}

\begin{proof}
Given $b\in B,$ let $b$ be the image of
$(g_n)_{n\in\omega}\in\prod_{n\in\omega} G_n.$
Then the homomorphism $\gamma:\Z^\omega\to\prod_{n\in\omega} G_n$
taking $(m_n)_{n\in\omega}$ to $(g_n^{m_n})_{n\in\omega}$
induces a homomorphism
$\gamma':\Z^\omega/\U\to(\prod_{n\in\omega} G_n)/\U,$
with which it forms a commuting square.
Hence the composite map $\Z^\omega\to \prod_{n\in\omega} G_n\to
\prod_{n\in\omega} G_n/\U\to B,$
which carries $(1,1,\dots)\in\Z^\omega$ to
$b,$ factors through $\Z^\omega/\U;$
so $b$ lies in a homomorphic image of that group.
\end{proof}

To see that this puts strong restrictions on
groups $B$ admitting such homomorphisms, note that every
slender abelian group, in particular, the infinite cyclic group,
has the property of {\em not} being a homomorphic image of a
nonprincipal countable ultrapower of $\Z.$
We will see wider classes of abelian groups with this
property in the next section.

Though this note emphasizes the separate
conditions that maps from infinite products yield
factorizations through finitely many ultraproducts, and that
the ultraproducts in all such factorizations are principal, let
us record how the above lemma allows one to combine results of the
former sort obtained in~\S\ref{S.gp_ultra} above,
and results of the latter sort for abelian groups, which will be
obtained in~\S\S\ref{S.ab}-\ref{S.ab_ac+cotor},
to give sufficient conditions
for all maps from a direct product of groups to factor
through finitely many projection maps.

\begin{theorem}\label{T.fin_many_i}
Let $B$ be a group with the property that for every homomorphism
from a direct product group,
\begin{equation}\begin{minipage}[c]{35pc}\label{d.f}
$f:\prod_{i\in I} G_i\to B$ such that the composite homomorphism
$\pi f:\prod_{i\in I} G_i\to B/Z(B)$ is surjective,
\end{minipage}\end{equation}
the map $\pi f$
factors through the projection to finitely many ultraproducts
of the~$G_i$ \textup{(}cf.~\S\ref{S.gp_ultra} above\textup{)}.

Suppose, moreover, that for every almost direct
factor $B'\neq Z(B)$ of $B,$ the group $B'/Z(B)$
contains at least one element $b$ which does
not lie in any homomorphic image therein
of a nonprincipal countable ultraproduct of copies of $\Z$
\textup{(}cf.~\S\S\ref{S.ab}-\ref{S.ab_Chase&&}
below\textup{)}.

Then for every homomorphism~\eqref{d.f} such that $\card(I)$ is
less than every uncountable measurable cardinal
\textup{(}if any such cardinals exist\textup{)},
the composite $\pi f:\prod_{i\in I} G_i\to B/Z(B)$
factors through the product of finitely many of the~$G_i.$
\end{theorem}

\begin{proof}
Given a homomorphism~\eqref{d.f} satisfying the indicated
bound on $\card(I),$ let us factor
$\pi f$ through a direct product
$\prod_{i\in I} G_i/\U_0\times\dots\times\prod_{i\in I} G_i/\U_{m-1},$
where $\U_0,\dots,\U_{m-1}$ are distinct ultrafilters on $I.$
Without loss of generality, we may assume that
each $\prod_{i\in I} G_i/\U_k$ has nontrivial image in $B/Z(B).$
Choosing a partition $I=J_0\cup\dots\cup J_{m-1}$
with $J_k\in\U_k,$ we get, by Lemma~\ref{L.B_S},
an almost direct decomposition of $B$ into subgroups $B_{J_k}.$
Now suppose one of our ultrafilters $\U_k$ were not principal.
By our assumption on the cardinality of $I,$ $\U_k$ is not
$\!\kappa\!$-complete for any uncountable measurable
cardinal $\kappa,$ hence by Lemma~\ref{L.countable},
\eqref{d.non-meas}$\!\implies\!$\eqref{d.omega},
$B_{J_k}/Z(B)$ satisfies the hypothesis of Lemma~\ref{L.ultra}.
But since $B_{J_k}$ is an almost direct factor of $B,$
by assumption $B_{J_k}/Z(B)$ has an element
$b$ whose properties contradict the conclusion of that lemma.
So, rather, every $\U_k$ must be principal, say
generated by a singleton $\{n_k\}\subseteq J_k.$
Hence our factorization through
$(\prod_{i\in I} G_i)/\U_0\times\dots\times
(\prod_{i\in I} G_i)/\U_{m-1}$ is in fact a
factorization through $G_{n_0}\times\dots\times G_{n_{m-1}}.$
\end{proof}

Quick examples of groups $B$ to which the above result applies are free
groups on more than one generator, and the infinite dihedral group.
Indeed, since both groups have trivial center, the ``$\!/Z(B)$\!''
in the statement can be ignored, and since neither has a nontrivial
direct product decomposition, it suffices to verify that each
has an element $b$ not contained in any
homomorphic image of a nonprincipal countable ultraproduct of
copies of $\Z.$
In a free group, every nontrivial abelian subgroup is infinite
cyclic, hence slender, so any nonidentity element can
serve as such a $b.$
In the dihedral group $D=\lang x,\,y\mid x^2=e=y^2\rang,$
the element $b=xy$ generates an infinite cyclic subgroup
which is its own centralizer,
again establishing the hypothesis of the theorem.
Another class of examples is noted in

\begin{corollary}[{to Lemma~\ref{L.ultra}}]\label{C.Sym}
Let $X$ be an infinite set, and $B$ a group of permutations
of $X$ having a cyclic subgroup $\lang b\rang$ whose action
on $X$ has exactly one infinite orbit
\textup{(}no restriction being assumed on the number
of finite orbits of $\lang b\rang).$
Then the centralizer of $b$ in $B$ admits a homomorphism
to $\Z$ taking $b$ to $1.$
Hence $B$ is not a homomorphic image of a nonprincipal
ultraproduct of a countable family of groups.

In particular, this is true if $B$ is the full symmetric
group on $X,$ or more generally, if for some filter $\F$ on $X$
not consisting entirely of cofinite subsets, $B$ is the group
of permutations of $X$ whose fixed sets belong to $\F.$
\end{corollary}

\begin{proof}
If two permutations $a$ and $b$ of a set $X$ commute, it is easy
to see that $a$ will carry orbits of $\lang b\rang$ to orbits
of $\lang b\rang,$ and, of course, the image orbits will have the
same cardinalities as the original orbits.
Hence, if $\lang b\rang$ has a unique infinite orbit $Y,$
then $a$ must carry $Y$ to itself;
and it is easy to verify that it must act on $Y$ by
some power $b^{n_a}$ of $b.$
The function $a\mapsto n_a$ now gives the
desired homomorphism of the centralizer of $b$ onto $\Z.$
Hence, every commutative subgroup of $B$ containing
$b$ admits a homomorphism onto $\Z,$ so as in the other examples
discussed above, $B$ is not a homomorphic image of a nonprincipal
ultraproduct of a countable family of groups.

Now if $\F$ is a filter on $X$ containing
a set $W$ which is not cofinite, we can take a countably
infinite subset $Y\subseteq X-W,$ and let $b$
be a permutation which
has $Y$ as an orbit, and fixes all other points of $X.$
This gives the final assertion of the corollary.
The full symmetric group on $X$ is the particular case where
$\F$ is the improper filter on $X.$
\end{proof}

(With a little more work, one can get a result similar to
the first paragraph of the above corollary under the weaker assumption
that $\lang b\rang$ has at least one but only finitely
many distinct infinite orbits, say
$\lang b\rang\,x_0,\dots, \lang b\rang\,x_{d-1}\subseteq X.$
In this case, for each $a$ centralizing $b,$ we
find that $a x_i = b^{n_{a,i}} x_{\pi_a(i)}$ $(0\leq i<d)$
for some permutation $\pi_a$ of $\{0,\dots,d-1\}$ and integers
$n_{a,0},\dots,n_{a,d-1}.$
It is then easy to verify
that the map $a\mapsto\sum_i n_{a,i}$ is a homomorphism
from the centralizer of $b$ to $\Z,$ which carries $b$ to~$d.)$

The next result, in contrast, gives a large class of groups that
{\em do} admit surjective homomorphisms from nonprincipal countable
ultraproducts.
The construction appears to be well-known, but I have not
been able to find a reference.

\begin{proposition}[reference?]\label{P.compact}
If a group $B$ admits a compact Hausdorff group topology,
then for any set $I$ and any ultrafilter $\U$ on $I,$ there
exists a homomorphism $B^I/\,\U\to B$ left-inverse to the
natural embedding $B\to B^I\to B^I/\,\U$ \textup{(}where the first
arrow is the diagonal map\textup{)}.

Hence, every group $B$ admitting a compact Hausdorff
group topology is a homomorphic
image of a nonprincipal countable ultraproduct of groups;
hence so is every homomorphic image of such a group.

These statements hold, more generally, with groups replaced
by the objects of any
variety of finitary algebras, in the sense of universal algebra.
\end{proposition}

\begin{proof}[Sketch of proof]
Fix a compact Hausdorff group topology on $B.$
Given $x\in B^I,$ let us associate to each $S\in\U$
the set $X_S = \{x_i\mid i\in S\}\subseteq B.$
These sets clearly have the finite intersection property,
hence so do their closures.
On the other hand, with the help of the definition of ultrafilter
and the Hausdorffness of our topology, it is easy to verify
that those closures can have no more than one common point.
Hence by compactness, the system
of sets $X_S$ must converge to a single point of $B.$
It is immediate that
the map associating to $x$ the limit point of this system depends
only on the image of $x$ in $B^I/\,\U,$
and so induces a map $B^I/\,\U\to B,$ and it is easy to
verify that this is a homomorphism with the asserted properties.

The statements in the second paragraph of the lemma clearly follow.
The final generalization holds by the same reasoning.
\end{proof}

By the above result, such a $B$ is a homomorphic image
of $B^I/\U$ for {\em every} ultrafilter $\U$ on every set $I.$
This suggests the following question, where for simplicity
we limit ourselves to $I=\omega.$

\begin{question}\label{Q.B,B'}
If $\U,$ $\U'$ are nonprincipal ultrafilters on $\omega,$
can every group $B$ which can be written as a homomorphic
image of an ultraproduct of groups with respect to $\U$
also be written as a homomorphic
image of an ultraproduct of groups with respect to $\U'$?
\end{question}

\begin{question}\label{Q.B,B',B''}
If the answer to Question~\ref{Q.B,B'}
is negative, is it at least true that for
any two ultrafilters $\U$ and $\U'$ on $\omega,$
there exists an ultrafilter $\U''$ on $\omega$
such that every group which can be written as a homomorphic
image of an ultraproduct of groups with respect to $\U$
{\em or} with respect to $\U'$ can be written as a homomorphic
image of an ultraproduct with respect to $\U''$?
\end{question}

If Question~\ref{Q.B,B',B''} has a positive answer, one
can deduce that the class of groups which can be written as
homomorphic images of nonprincipal countable ultraproducts of groups is
closed under finite direct products.

Proposition~\ref{P.compact} also leads one to wonder whether every
group $B$ which can be written as a homomorphic image of
a nonprincipal countable ultraproduct of groups can in fact
be written as a homomorphic image of
a nonprincipal countable ultrapower $B^\omega/\U$ of itself,
via a left inverse to the natural embedding $B\to B^\omega/\U.$
The answer is negative; we shall see in the second paragraph
after Lemma~\ref{L.ab_cotor_not_cp}
that there exist abelian groups for which this is not true.

Let us note a couple of groups $B$ for which the results
of this section do not, as far as I can see, give us any information.

\begin{question}\label{Q.B_gp+}
Can either of the following groups be written as
a homomorphic image of a nonprincipal
ultraproduct of a countable family of groups?\\[.3em]
\textup{(i)} An infinite finitely generated Burnside group?\\[.3em]
\textup{(ii)} The group of those permutations of an infinite
set that move only finitely many elements?
\textup{(}Contrast Corollary~\ref{C.Sym}.\textup{)}
\end{question}

Let us also record, since we know no counterexample,

\begin{question}\label{Q.ultra}
Is the converse to Lemma~\ref{L.ultra} true?
That is, if $\U$ is an ultrafilter on $\omega,$ and
$B$ is a group such that every $b\in B$ lies in a homomorphic
image within $B$ of $\Z^\omega/\U,$ must $B$ be a homomorphic
image of an ultraproduct group $\prod_{i\in\omega} G_i/\U$?
\end{question}

A positive answer seems extremely unlikely.
It would imply, in particular, that every torsion
group was such a homomorphic image for every $\U.$
(So it would imply positive answers to both parts of
Question~\ref{Q.B_gp+}.)

We remark that the results we have obtained
so far show that the two sorts
of properties of an object $B$ that we are considering in this note --
(a)~that surjective homomorphisms from direct products onto $B$
yield factorizations through finitely many ultraproducts, and
(b)~that when one has such a factorization, and the index-set
of the product is countable, the ultraproducts involved
must be principal -- are independent, for groups.
Theorem~\ref{T.fin_many_i} gave us examples satisfying both~(a)
and~(b), such as the
free group on more than one generator, and the infinite dihedral group.
Any infinite direct product of free
groups on more than one generator will still satisfy~(b)
(since a homomorphism from a nonprincipal countable ultraproduct
group {\em into} such a product will have trivial composite with the
projection onto each factor, hence must be trivial), but will fail to
satisfy~(a), by virtue of {\em being} an infinite direct product.
Examples satisfying~(a) but not~(b) are given by groups
satisfying the hypotheses of both
Corollary~\ref{C.one_U} and Proposition~\ref{P.compact};
for instance, finite simple groups.
Finally,
infinite direct products of such examples satisfy neither~(a) nor~(b).

(Incidentally, if a map $\prod_{i\in I} G_i\to B$ factors
$\prod_{i\in I} G_i\to (\prod_{i\in I} G_i)/\U_0\times\dots\times
(\prod_{i\in I} G_i)/\U_{m-1}\to B,$ one or more of the factors
$(\prod_{i\in I} G_i)/\U_k$ may be irrelevant to the
factorization, i.e., may map trivially to $B.$
In condition (b) in the above discussion,
we understand the phrase ``the ultraproducts involved''
to exclude such ``irrelevant'' factors; if we did not, (b) could
never hold.)

\section{Abelian groups}\label{S.ab}

We have seen that in the study of homomorphisms on products of
nonabelian groups, the analogous questions for
abelian groups are important.
We now turn to that case.

Although, as just noted,
the two sorts of condition we are interested in
are independent for nonabelian groups, we shall find
that this is not true of the corresponding conditions on abelian groups.

First, some notation, language, and basic observations.

\begin{definition}\label{D.ab}
In \S\S\ref{S.ab}-\ref{S.ab_Chase&&},
we shall use additive notation in abelian groups.

In groups $\Z^\omega,$ $(\Z/p\Z)^\omega,$ etc.,
we shall write $\delta_n$ $(n\in\omega)$ for the element having $1$
in the $\!n\!$-th position and $0$ in all other positions.

An abelian group $B$ is called {\em slender} if
it is torsion-free, and every homomorphism $f:\Z^\omega\to B$
annihilates all but finitely many of the $\delta_n.$
\end{definition}

The above definition of a slender abelian group is standard, but
the condition that $B$ be torsion-free is redundant: no
$B$ with torsion satisfies the condition on homomorphisms.
For in such a $B,$ we can choose an element $b$ of prime order $p,$
define the homomorphism $\bigoplus_{n\in\omega} \Z/p\Z\to\lang b\rang$
taking each $\delta_n$ to $b,$ extend this,
by linear algebra over the field $\Z/p\Z,$ to a homomorphism
$(\Z/p\Z)^\omega\to \lang b\rang,$ and precompose with
the natural map $\Z^\omega\to(\Z/p\Z)^\omega,$ to get a map
$\Z^\omega\to B$ that does not annihilate any $\delta_n.$

The condition of slenderness is stronger than it looks.
Indeed, our statement of that condition
in~\S\ref{S.intro} implicitly incorporated the following striking
complementary fact.
\begin{equation}\begin{minipage}[c]{35pc}\label{d.0_if}
(\cite[fact~(f) on p.159]{Fuchs2})\ \ If $B$ is a slender
abelian group, then the only homomorphism $f:\Z^\omega\to B$ which
annihilates all the elements $\delta_n$ $(n\in\omega)$ is $0.$
\end{minipage}\end{equation}

We can now prove

\begin{proposition}\label{P.slender}
The following conditions on an abelian group $B$ are equivalent.
\begin{equation}\begin{minipage}[c]{35pc}\label{d.ab_not_fin_U}
There exists a surjective homomorphism $f:\prod_{i\in I} A_i\to B$
from the direct product of a family of abelian groups to $B,$
which does not factor through the natural map from $\prod_{i\in I} A_i$
to the direct product of finitely many
countably complete ultraproducts of the $A_i.$
\end{minipage}\end{equation}
\begin{equation}\begin{minipage}[c]{35pc}\label{d.ab_not_fin_n}
There exists a surjective homomorphism $f:\prod_{n\in\omega} A_n\to B$
from the direct product of a countable family of abelian groups to $B,$
which does not factor through the projection of
$\prod_{n\in\omega} A_n$ to the direct
product of finitely many ultraproducts \textup{(}principal
or nonprincipal\textup{)} of the $A_n.$
\end{minipage}\end{equation}
\begin{equation}\begin{minipage}[c]{35pc}\label{d.not_slender}
$B$ is not slender.
\end{minipage}\end{equation}

These are also equivalent to the variants
of conditions~\eqref{d.ab_not_fin_U} and~\eqref{d.ab_not_fin_n}
without the assumption that $f$ be surjective.
\end{proposition}

\begin{proof}
We start with the final sentence.
\eqref{d.ab_not_fin_n} and~\eqref{d.ab_not_fin_U} certainly imply
the corresponding statements without the condition of surjectivity.
Conversely (as noted in~\S\ref{S.intro}),
if we have an example of either of those conditions
minus the surjectivity restriction, we can get one satisfying that
condition by passing from the given
map $\prod_{i\in I} A_i\to B$ to the obvious surjective map
$B\times \prod_{i\in I} A_i\to B.$

Let us now show that
\eqref{d.not_slender}$\!\implies\!$\eqref{d.ab_not_fin_n}%
$\!\implies\!$\eqref{d.ab_not_fin_U}$\!\implies\!$\eqref{d.not_slender}.

Given~\eqref{d.not_slender}, take a map $f:\Z^\omega\to B$ witnessing
the failure of slenderness, i.e., carrying infinitely many
of the $\delta_n$ to nonzero values.
If $f$ factored through a product of finitely many ultrapowers,
$\Z^\omega/\U_0\times\dots\times\Z^\omega/\U_{m-1},$
then the only elements
$\delta_n$ which could have nonzero image under $f$ would be those
such that one of the $\U_k$ was the principal ultrafilter generated
by $\{n\},$ of which there can be at most finitely many.
So there is no such factorization,
so $f$ witnesses~\eqref{d.ab_not_fin_n}
(in its version without the hypothesis of surjectivity).

Clearly,~\eqref{d.ab_not_fin_n}$\!\implies\!$\eqref{d.ab_not_fin_U}.

Given $f$ as in~\eqref{d.ab_not_fin_U}, we shall
prove~\eqref{d.not_slender} by considering two cases.
First suppose that $f$ can be factored through a product
of finitely many ultraproducts
$\prod_{i\in I} A_i/\U_0\times\dots\times\prod_{i\in I} A_i/\U_{m-1},$
but that not all the $\U_k$ can be taken countably complete.
Note that $f$ is the {\em sum} of homomorphisms $f_k$
$(k=0,\dots,m-1)$ that factor
through the respective ultraproducts $\prod_{i\in I} A_i/\U_k,$ and
we can drop from this sum, and hence from our
factorization, any factors $\prod_{i\in I} A_i/\U_k$
such that $f_k$ is zero.
Hence for some $k$ with $\U_k$ {\em not} countably complete, we must
have a {\em nonzero} homomorphism $f_k:\prod_{i\in I} A_i/\U_k\to B.$
The statement that $\U_k$ is not countably complete is equivalent
to saying that
there exists a partition $I=J_0\cup\ldots\cup J_n\cup\dots$
such that none of the $J_n$ lie in $\U_k;$ in other words,
such that for each $n,$ $f_k\,|\prod_{i\in J_n} A_i=0.$
If we regard $f_k$ as a map
$\prod_{n\in\omega}(\prod_{i\in J_n} A_i)\to B,$
the fact that it is nonzero means that we can choose an element
$(x_n)_{n\in\omega}\in\prod_{n\in\omega}(\prod_{i\in J_n} A_i)$
which $f_k$ sends to a nonzero element of $B,$
though we know that it takes each $x_n$ to $0.$
Using this element $(x_n)_{n\in\omega},$
let us construct a map $\Z^\omega\to B$
by taking each $(d_n)_{n\in\omega}\in\Z^\omega$ to
$(d_n x_n)_{n\in\omega},$ and applying $f_k$ to this $\!\omega\!$-tuple.
This gives a homomorphism $\Z^\omega\to B$
which is zero on each $\delta_n,$ but not on $(1,\dots,1,\dots).$
Thus, by~\eqref{d.0_if}, $B$ is not slender.
(Alternatively,
we can get a direct contradiction to the definition of slenderness
by choosing $f_k$ and $(x_n)_{n\in\omega}$ as above, and
mapping $(d_n)_{n\in\omega}\in\Z^\omega$ to
$f_k((\sum_{m<n}d_m) x_n)_{n\in\omega}\in B.)$

There remains the case where $f$ cannot be factored through any
product of finitely many ultraproducts of the $A_i.$
Then the filter $\F$ of subsets $S\subseteq I$ such that
$f$ can be factored through $\prod_{i\in S} A_i$
is not a finite intersection of ultrafilters, so
by~\eqref{d.cap_ultra} there
exists a partition $I=J_0\cup\ldots\cup J_n\cup\dots$
such that $I-J_n\notin\F$ for all $n;$ in other words,
such that each $\prod_{i\in J_n} A_i\subseteq\prod_{i\in I} A_i$
has nonzero image under $f.$
Choosing an $x_n$ in each $\prod_{i\in J_n} A_i$ with nonzero image,
we construct as in the preceding case a homomorphism $\Z^\omega\to B.$
This time, that homomorphism will be nonzero
on every $\delta_n,$ showing that $B$ does not satisfy
the definition of slenderness.
\end{proof}

Slender abelian groups have been precisely characterized
(\cite{Nunke}, \cite[Proposition~95.2]{Fuchs2}):
they are the abelian groups which have no torsion elements,
and contain no embedded copies of either $\Q,$ or the group
of $\!p\!$-adic integers for any prime $p,$ or $\Z^\omega.$

In the statement of the above proposition, note that
condition~\eqref{d.ab_not_fin_U} is formally weaker
than~\eqref{d.ab_not_fin_n} in two ways: it allows
an arbitrary index set $I,$ and it excludes factorization
only through {\em countably complete} ultraproducts (which in the
context $I=\omega$ of~\eqref{d.ab_not_fin_n} would mean
principal ultraproducts, i.e., the given groups $A_n).$
Since~\eqref{d.ab_not_fin_U} and~\eqref{d.ab_not_fin_n} are
equivalent, they are also equivalent to two intermediate
conditions: the one obtained from~\eqref{d.ab_not_fin_n} by
replacing ``ultraproducts (principal or nonprincipal)'' by
``principal ultraproducts'', and the one obtained
from~\eqref{d.ab_not_fin_U} by deleting the
words ``countably complete''.

We can deduce from these observations that
the two sorts of conditions on an abelian group
$B$ that we are interested in -- namely,
(a)~that maps to $B$ from infinite direct products factor through
finitely many ultraproducts, and (b)~that in the case of
a countable product, if we have a such a factorization, the ultrafilters
involved are all principal -- are not independent;
precisely, that (a)~implies~(b).
Indeed, (a) is equivalent to the negation of
the version of~\eqref{d.ab_not_fin_U} without the ``countably
complete'' condition, which by the above observations is
equivalent to the negation of the version of~\eqref{d.ab_not_fin_n}
in which the ultrafilters are assumed principal, i.e., the
statement that every homomorphism from a countable product into
$B$ factors through finitely many $A_n,$ which clearly entails~(b).
Bringing in~\eqref{d.not_slender}, we see that
both~(a) and (a)$\wedge$(b) are equivalent to slenderness.

On the other hand, the three cases not
excluded by the implication (a)$\!\implies\!$(b) all do occur.
Slender groups, such as $\Z,$ satisfy both~(a) and~(b).
An infinite direct product of nontrivial slender groups,
e.g., $\Z^\omega,$ satisfies~(b) but not~(a).
Finally, any nonprincipal countable
ultraproduct of nontrivial abelian groups will not
satisfy~(b), hence, since (a)$\!\implies\!$(b), it will satisfy neither.

Having characterized the abelian groups $B$ that satisfy~(a),
it remains to characterize the larger class satisfying~(b).
As preparation, we shall first study the abelian groups that
are homomorphic images of a {\em single}
nonprincipal countable ultraproduct.
We will need a few more definitions from the theory of
infinite abelian groups.

\begin{definition}[\cite{Fuchs1},~\cite{Rotman}]\label{D.alg_cp,cotor}
A subgroup $B$ of an abelian group $A$ is called {\em pure} if
for every positive integer $n,$ $B\cap nA=nB.$

An abelian group $B$ is said to be {\em algebraically compact} if
for every overgroup $A\supseteq B$ in which $B$ is pure, $B$ is
a direct summand in $A;$ equivalently \cite[Theorem~38.1]{Fuchs1},
if for every set $X$ of group equations in constants
from $B$ and $\!B\!$-valued variables, such that every finite subset
of $X$ has a solution in $B,$ the whole set $X$ has a solution in $B.$

An abelian group $B$ is said to be {\em cotorsion} if
for every overgroup $A\supseteq B$ such that $A/B$ is torsion-free
\textup{(}a stronger condition than $B$ being pure
in $A),$ $B$ is a direct summand in $A.$
\end{definition}

Of the two definitions of algebraic compactness quoted above,
the first is the one commonly used.
I include the second because it motivates the name of the condition.
The theorem
cited for their equivalence establishes several other diverse
conditions as also being equivalent to algebraic compactness; below,
I shall pull these out of a hat as needed.

The cotorsion abelian groups clearly include the
algebraically compact abelian groups.
In fact, they are precisely the {\em homomorphic images} of such groups
\cite[Proposition~54.1]{Fuchs1}, a fact called on
in condition~\eqref{d.cotorsion} in the next result.

\begin{proposition}\label{P.cotor}
For $B$ an abelian group, the following conditions are equivalent.
\begin{equation}\begin{minipage}[c]{35pc}\label{d.ab_im_ultra}
There exists a family of abelian groups $(A_i)_{i\in I}$ and
a non-countably-complete ultrafilter $\U$ on $I$
such that $B$ is a homomorphic image of the ultraproduct
$\prod_{i\in I} A_i/\U.$
\end{minipage}\end{equation}
\begin{equation}\begin{minipage}[c]{35pc}\label{d.ab_im_ctbl_ultra}
There exists a countable family of abelian groups $(A_n)_{n\in\omega}$
and a nonprincipal ultrafilter $\U$ on $\omega$
such that $B$ is a homomorphic image of the ultraproduct
$\prod_{n\in\omega} A_n/\U.$
\end{minipage}\end{equation}
\begin{equation}\begin{minipage}[c]{35pc}\label{d.ab_im_redpr}
There exists a countable family of abelian groups $(A_n)_{n\in\omega}$
and a filter $\F$ on $\omega$ which is not contained in any
principal ultrafilter \textup{(}i.e., which
satisfies $\bigcap_{S\in\F} S=\emptyset),$ such
that $B$ is a homomorphic image of $\prod_{n\in\omega} A_n/\F.$
\end{minipage}\end{equation}
\begin{equation}\begin{minipage}[c]{35pc}\label{d.ab_im_frechet}
There exists a countable family of abelian groups $(A_n)_{n\in\omega}$
such that $B$ is a homomorphic image of the reduced product
$(\prod_{n\in\omega} A_n)/\bigoplus_{n\in\omega} A_n.$
\end{minipage}\end{equation}
\begin{equation}\begin{minipage}[c]{35pc}\label{d.ab_im_cp}
$B$ is a homomorphic image of an abelian group $C$ admitting
a compact Hausdorff group topology.
\end{minipage}\end{equation}
\begin{equation}\begin{minipage}[c]{35pc}\label{d.cotorsion}
$B$ is a cotorsion abelian group; i.e., a homomorphic image of an
algebraically compact abelian group.
\end{minipage}\end{equation}
\end{proposition}

\begin{proof}
We shall show
\eqref{d.ab_im_ultra}$\!\implies\!$%
\eqref{d.ab_im_ctbl_ultra}$\!\implies\!$%
\eqref{d.ab_im_redpr}$\!\implies\!$%
\eqref{d.ab_im_frechet}$\!\implies\!$%
\eqref{d.cotorsion}$\!\implies\!$%
\eqref{d.ab_im_cp}$\!\implies\!$%
\eqref{d.ab_im_ultra}.

In the situation of~\eqref{d.ab_im_ultra}, the fact that $\U$ is
not countably complete implies that we can find a partition
$I=\bigcup_{n\in\omega} J_n$ such that no $J_n$ belongs to $\U.$
Let us again write
$\prod_{i\in I} A_i=\prod_{n\in\omega}(\prod_{i\in J_n} A_i).$
As in the proof of Lemma~\ref{L.countable}, if we let
$\U'=\{S\subseteq\omega\mid \bigcup_{n\in S} J_n\in\U\},$
we find that $\U'$ is a nonprincipal ultrafilter on $\omega,$
yielding a factorization of the map from our
product group to our original ultraproduct as
$\prod_{i\in I} A_i\to\prod_{n\in\omega}(\prod_{i\in J_n} A_i)/\U'\to
\prod_{i\in I} A_i/\U.$
Since $B$ is a homomorphic image of $\prod_{i\in I} A_i/\U,$
it is a homomorphic image of the factoring object,
proving~\eqref{d.ab_im_ctbl_ultra}.

We get~\eqref{d.ab_im_ctbl_ultra}$\!\implies\!$\eqref{d.ab_im_redpr}
by taking $\F=\U.$

Given~\eqref{d.ab_im_redpr},
note that since the filter $\F$ on $\omega$ is not
contained in a principal ultrafilter, it contains the
complement of every singleton, hence it contains the
{\em Fr\'{e}chet} filter $\mathcal{C}$
of complements of finite sets.
So the quotient map
$\prod_{n\in\omega} A_n\to\prod_{n\in\omega} A_n/\F$
factors through $\prod_{n\in\omega} A_n/\mathcal{C}=
(\prod_{n\in\omega} A_n)/\bigoplus_{n\in\omega} A_n,$
giving~\eqref{d.ab_im_frechet}.

Given~\eqref{d.ab_im_frechet}, we call on \cite[Corollary~42.2]{Fuchs1}
which says that every group of the form
$(\prod_{n\in\omega} A_n)/\bigoplus_{n\in\omega} A_n$
is algebraically compact, yielding~\eqref{d.cotorsion}.

For the step \eqref{d.cotorsion}$\!\implies\!$\eqref{d.ab_im_cp}, we
call on \cite[Theorem~38.1]{Fuchs1} (or on \cite[Theorem~7.42]{Rotman})
which, among the equivalent conditions for an abelian group to be
algebraically compact, includes that of being a {\em direct summand}
in an abelian group that admits a compact Hausdorff group topology.
So an algebraically compact abelian group is, in particular, a
{\em homomorphic image}
of an abelian group admitting such a topology, hence so is
any homomorphic image of an algebraically compact group.

Finally, by Proposition~\ref{P.compact} above, any abelian group
$A$ admitting a compact Hausdorff group topology
can be written as a homomorphic image of its ultrapower $A^I/\U$
for any ultrafilter $\U$ on any set $I.$
So choosing a $\U$ which is not countably complete (e.g., any
nonprincipal ultrafilter on $I=\omega),$ we
get \eqref{d.ab_im_cp}$\!\implies\!$\eqref{d.ab_im_ultra}.
\end{proof}

We note that for a nonzero abelian group $B,$
the equivalent conditions of Proposition~\ref{P.cotor}
imply those of Proposition~\ref{P.slender}.
Indeed, thinking in terms of the conditions (a) and (b) that we have
been discussing, if we write (b$\!_1\!$) for the case of~(b)
where there is only a single ultraproduct involved
(i.e., the condition that if there exists a nonzero homomorphism
from an ultraproduct group $\prod_{n\in\omega} A_n/\U$
onto $B,$ then the ultrafilter $\U$ is principal), then we have
$(\r{a}){\implies}(\r{b}){\implies}(\r{b}_1),$
so $\neg(\r{b}_1){\implies}\neg(\r{a});$
moreover, we see that for $B\neq\{0\},$~\eqref{d.ab_im_ctbl_ultra}
is equivalent to $\neg(\r{b}_1),$ while we have previously
noted that the conditions of Proposition~\ref{P.slender} are equivalent
to $\neg(\r{a}).$
(Alternatively, it not hard to see directly that for $B\neq\{0\},$ an
example witnessing~\eqref{d.ab_im_ctbl_ultra}
also witnesses~\eqref{d.ab_not_fin_U}.)
Choosing the equivalent conditions of the two propositions that
have standard names, these observations say that for $B\neq\{0\},$
\eqref{d.cotorsion}$\!\implies\!$\eqref{d.not_slender};
in other words, no nonzero cotorsion abelian group is slender.
Since the class of cotorsion abelian groups is closed under
homomorphic images, this in fact gives

\begin{corollary}\label{C.slender_vs_cotorsion}
No cotorsion abelian group has a nonzero slender homomorphic image.\qed
\end{corollary}

Corollary~\ref{C.slender_vs_cotorsion}
allows us to apply Theorem~\ref{T.fin_many_i} to many variants
of the examples immediately following it.
For instance, one of those was the infinite
dihedral group, i.e., the semidirect product arising from
the natural action of $\{\pm 1\}$ on the slender group $\Z.$
I claim that we can replace $\Z$ in that example by
any abelian group $A$
without $\!2\!$-torsion that has $\Z$ as a homomorphic image;
for instance, $\Z^\omega,$ or $\Z\times\Z/n\Z$ for any odd $n.$
Indeed, taking a homomorphism from such an abelian group $A$ onto $\Z,$
and any $b\in A$ that maps to a generator of $\Z$ under that
homomorphism, we see from the above corollary that no subgroup $B$
of $A$ containing $b$ is cotorsion, equivalently,
by Proposition~\ref{P.cotor}, that no such subgroup $B$
satisfies~\eqref{d.ab_im_ctbl_ultra}; hence $b$ satisfies
the condition of the second paragraph of Theorem~\ref{T.fin_many_i}.
(The assumption that $A$ has no $\!2\!$-torsion keeps the center of the
semidirect product trivial, to avoid complicating our considerations.)
In \S\ref{S.ab_ac+cotor} we will obtain
more information on which abelian groups are cotorsion.

We can now answer the question of which abelian groups $B$
have the property we called~(b) in our earlier discussion,
namely, that any map $f$ from a countable direct product
of abelian groups $A_n$ onto $B$ which factors through finitely many
ultrafilters in fact factors through the projection
to the product of finitely many of the $A_n.$
We shall see that this is true if and only if $B$ contains
no nontrivial cotorsion subgroup.
Although the class of cotorsion abelian groups is
difficult to describe
exactly, a simple criterion is known for an abelian group to be
{\em cotorsion-free}, i.e., to contain no nontrivial cotorsion subgroup:
It is that the group be torsion-free, and contain no copy
of the additive group of $\Q,$ nor of the $\!p\!$-adic integers
for any prime $p$ \cite[Theorem~2.4 (1)$\!\implies\!$(4)]{MD+RG}.
(So it is like the condition characterizing slenderness, but
without the exclusion of subgroups isomorphic to $\Z^\omega.)$
This condition is also equivalent to that of
containing no nonzero algebraically compact subgroup:
it implies the latter because every algebraically compact group
is cotorsion, while the reverse implication holds because
$\Q,$ and the groups of $\!p\!$-adic integers, and all finite
abelian groups, are algebraically compact.
As is usual in this note, the statement below
will be the contrapositive of the version suggested by this discussion.

\begin{theorem}\label{T.not_prin}
The following conditions on an abelian group $B$ are equivalent.
\begin{equation}\begin{minipage}[c]{35pc}\label{d.ab_I_not_c_cplt}
There exist a set $I,$ a family of abelian groups $(A_i)_{i\in I},$
and a surjective homomorphism $f:\prod_{i\in I} A_i\to B$ such that $f$
factors through the product of finitely many ultraproducts
$\prod_{i\in I} A_i/\U_k,$ but does not factor through
the product of finitely many {\em countably complete} ultraproducts.
\end{minipage}\end{equation}
\begin{equation}\begin{minipage}[c]{35pc}\label{d.ab_omega_not_prin}
There exist a countable
family $(A_n)_{n\in\omega}$ of abelian groups and
a surjective homomorphism $f:\prod_{n\in \omega} A_n\to B$ such that $f$
factors through the product of finitely many ultraproducts
$\prod_{n\in\omega} A_n/\U_k,$ but does not factor through
the product of finitely many of the $A_n.$
\end{minipage}\end{equation}
\begin{equation}\begin{minipage}[c]{35pc}\label{d.ab_has_cotorsion}
$B$ has a nontrivial cotorsion subgroup; equivalently \textup{(}by
the result from \cite{MD+RG} quoted above\textup{)}, $B$ either
has nonzero elements of finite order, or contains a
copy of the additive group of $\Q,$ or contains a copy of the additive
group of the $\!p\!$-adic integers for some prime $p;$
equivalently, $B$ has a nontrivial algebraically compact subgroup.
\end{minipage}\end{equation}

These conditions are also equivalent to the variants
of~\eqref{d.ab_omega_not_prin} and~\eqref{d.ab_I_not_c_cplt}
without assumption that $f$ be surjective.
\end{theorem}

\begin{proof}
The equivalence
of~\eqref{d.ab_I_not_c_cplt} and~\eqref{d.ab_omega_not_prin}
to the corresponding conditions without the assumption of
surjectivity is seen as in the first paragraph of the
proof of Proposition~\ref{P.slender}.
We shall use those variants
in proving \eqref{d.ab_has_cotorsion}$\!\implies\!$%
\eqref{d.ab_omega_not_prin}$\!\implies\!$%
\eqref{d.ab_I_not_c_cplt}$\!\implies\!$%
\eqref{d.ab_has_cotorsion}.

Assuming~\eqref{d.ab_has_cotorsion}, let $C\subseteq B$ be a
nonzero cotorsion subgroup.
Then our earlier result
\eqref{d.cotorsion}$\!\implies\!$\eqref{d.ab_im_ctbl_ultra}
gives a surjective homomorphism $\prod_{n\in\omega} A_n/\U\to C$
for a family of abelian groups $A_n$ and a nonprincipal
ultrafilter $\U,$ which we regard as a nonzero homomorphism into $B.$
Since $\U$ is not principal, $f$ annihilates each of the $A_n,$
so it cannot be factored through the product of finitely many
of these, giving~\eqref{d.ab_omega_not_prin}.

Clearly,
\eqref{d.ab_omega_not_prin}$\!\implies\!$\eqref{d.ab_I_not_c_cplt},
since the countably complete ultrafilters on $\omega$ are the
principal ultrafilters.

Assuming~\eqref{d.ab_I_not_c_cplt}, let $f:\prod_{i\in I} A_i\to B$
be a homomorphism that factors through a product of ultraproducts
$\prod_{i\in I} A_i/\U_0\times\dots\times\prod_{i\in I} A_i/\U_{m-1},$
but not through such a product in which all the
$\U_k$ are countably complete.
As noted in the proof of Proposition~\ref{P.slender},
the given factorization is equivalent to an expression of $f$
as the sum of maps that factor
$\prod_{i\in I} A_i\to\prod_{i\in I} A_i/\U_k\to B,$ and
if any of these maps are zero,
we can drop them, leaving a factorization with all these maps
nonzero, and which, by choice of $f$ must
still have at least one with $\U_k$ not countably complete.
So there exists a nonzero map $g:\prod_{i\in I} A_i/\U\to B$ for
some non-countably-complete ultrafilter $\U$ on $I.$
Our earlier result
\eqref{d.ab_im_ultra}$\!\implies\!$\eqref{d.cotorsion}
now tells us that the nonzero image of $g$ is a
cotorsion submodule of $B,$ proving~\eqref{d.ab_has_cotorsion}.
\end{proof}

We have not yet said much about algebraically compact groups, except
that the cotorsion groups are their homomorphic images.
We record

\begin{lemma}\label{L.ab_ultra_retr}
The following conditions on an abelian group $B$ are equivalent.
\begin{equation}\begin{minipage}[c]{35pc}\label{d.ab_oA_retr}
For every proper filter $\F$ on a nonempty set $I,$ the
natural embedding $B\to B^I/\F$ has a left inverse.
\end{minipage}\end{equation}
\begin{equation}\begin{minipage}[c]{35pc}\label{d.ab_oE_retr}
There exists a nonprincipal ultrafilter $\U$ on $\omega$
such that the natural embedding $B\to B^\omega/\U$ has a left inverse.
\end{minipage}\end{equation}
\begin{equation}\begin{minipage}[c]{35pc}\label{d.ab_alg_cp}
$B$ is algebraically compact.
\end{minipage}\end{equation}
\end{lemma}

\begin{proof}
For any filter $\F$ on a set $I,$
the natural embedding $B\to B^I/\F$ is easily seen to be pure,
so the definition of algebraic compactness
gives \eqref{d.ab_alg_cp}$\!\implies\!$\eqref{d.ab_oA_retr}.
Clearly, \eqref{d.ab_oA_retr}$\!\implies\!$\eqref{d.ab_oE_retr}.

To show that \eqref{d.ab_oE_retr}$\!\implies\!$\eqref{d.ab_alg_cp},
we use the result \cite[third sentence of \S2]{Eklof}, that a
nonprincipal countable ultrapower of any abelian group $B$
is algebraically compact.
Hence~\eqref{d.ab_oE_retr} implies that $B$ is a direct
summand in an algebraically compact abelian group, from which one
easily sees that it itself is algebraically compact.
\end{proof}

Since the cotorsion abelian groups are the homomorphic images of the
algebraically compact ones, the above result shows that the analog
of Question~\ref{Q.B,B'} has a positive answer for abelian groups.
(This can also be seen from the proof of Proposition~\ref{P.cotor},
where the closing step
\eqref{d.ab_im_cp}$\!\implies\!$\eqref{d.ab_im_ultra}
allows us to choose $\U$ essentially arbitrarily.)

Another interesting necessary and sufficient
condition for $B$ to be algebraically compact, obtained
(in the more general context of modules) as \cite[Theorem~7.1(vi)]{J+L},
is that for every
set $I,$ the summation map $\bigoplus_{i\in I} B\to B$ extend
to a map $B^I\to B.$

\section{More on algebraically compact and cotorsion abelian groups}\label{S.ab_ac+cotor}

The distinction between the class of cotorsion abelian groups and
its subclass, the algebraically compact abelian groups, is a subtle one.
It follows from the
definitions that every cotorsion abelian group $B$ that is torsion-free
is algebraically compact \cite[Corollary~54.5]{Fuchs1}.
The only example I have found in the literature of a
cotorsion abelian group that is not algebraically compact,
that of \cite[Proposition~7.48(ii)]{Rotman}, is
described as an $\r{Ext}$ of other groups, rather than explicitly.
(It is known that for any abelian groups $A$ and $A',$
$\r{Ext}(A,A')$ is cotorsion
\cite[Theorem~54.6]{Fuchs1}, \cite[Corollary~7.47]{Rotman}.)
Let us begin this section by constructing a more explicit example.

We will use the characterization of an algebraically compact
abelian group as an abelian group $B$ such that
whenever a system of equations has the property that all its
finite subsystems have solutions in $B,$
then the whole system has such a solution.
An easy example of an infinite system of
equations is the following, where $x_0$ is a
given element of $B,$ and $x_1,\dots,x_n,\dots$ are to be found.
\begin{equation}\begin{minipage}[c]{35pc}\label{d.ab_x0...}
$x_0=p\,x_1,\quad x_1=p\,x_2,\quad\dots,\quad
x_{n-1}=p\,x_n,\quad\dots\,.$
\end{minipage}\end{equation}

The necessary condition for algebraic compactness that
this system yields is

\begin{lemma}\label{L.ab_divisible}
If $B$ is an algebraically compact group and $p$ a prime,
then the subgroup $B'=\bigcap_{n\in\omega} p^n\,B\subseteq B$ is
$\!p\!$-divisible, i.e., satisfies $p\,B'=B'.$
\end{lemma}

\begin{proof}
Suppose $x_0\in B'.$
Let us fix $n\geq 0,$ and choose $x_n\in B$ such that
$x_0=p^n x_n.$
If we now let $x_m=p^{n-m} x_n$ for $0<m<n,$ we see that
$x_0,\dots,x_n$ satisfy the first $n$ equations of~\eqref{d.ab_x0...}.
Since we can do this for any $n,$
every finite subfamily of~\eqref{d.ab_x0...} has a solution, so
algebraic compactness implies that we can choose
$x_1,\dots,x_n,\dots$ satisfying the full set of equations.
For such $x_1,\dots,x_n,\dots$ we see that
$x_1$ also belongs to $B';$ so $x_0\in p\,B',$ as required.
\end{proof}

(It is also not hard to prove the above lemma from the
definition of algebraic compactness in terms of pure extensions:
Given algebraically compact $B,$ and $x_0\in B',$
let $B^+$ be the extension of $B$ gotten by adjoining new generators
$x_1,\dots,x_n,\dots$ and the relations~\eqref{d.ab_x0...}.
It is straightforward to show that $B$ embeds in $B^+,$ and
from the fact that $x_0\in B',$
one can deduce that $B$ is pure in $B^+.$
Hence the definition of algebraic compactness says that there
exists a retraction of $B^+$ onto $B,$ i.e., a solution
to~\eqref{d.ab_x0...} in $B;$ hence, as above, $x_0=p\,x_1\in p\,B'.)$

So let us try to construct a cotorsion abelian group $B$
with an element that we force to lie in $B',$ without
creating any apparent reason why it should lie in $p\,B'.$
To do this, let $\Z_p$ denote the additive group
of $\!p\!$-adic integers, which is algebraically compact
by Proposition~\ref{P.compact} and Lemma~\ref{L.ab_ultra_retr};
within its countable power $\Z_p^\omega,$ let $\delta_n$ be,
as usual, the element with $1$ in the $\!n\!$-th coordinate and $0$
in all others; and for a first try, let $B$ be the factor group of
$\Z_p^\omega$ by the subgroup generated by the elements
\begin{equation}\begin{minipage}[c]{35pc}\label{d.*d0=p^n*dn}
$\delta_0-p^n\,\delta_n$ $(n\in\omega).$
\end{minipage}\end{equation}
Letting $x$ be the image of $\delta_0$ in $B,$ we
clearly have $x\in B'.$

But this group is messy, making it hard to see whether
some $y\in B'$ might satisfy $x=p\,y.$
It becomes nicer if we impose~\eqref{d.*d0=p^n*dn} as $\!\Z_p\!$-module
relations rather than just as additive group relations.
If we then change coordinates in $\Z_p^\omega,$ so that
the elements $\delta_n-p\,\delta_{n+1}$ become the new $\delta_n$
(namely, we map $(a_n)_{n\in\omega}$ to
$(\sum_{m\leq n} p^{n-m} a_m)_{n\in\omega}),$
the resulting construction takes the form shown in the next lemma.

\begin{lemma}\label{L.ab_cotor_not_cp}
Let $p$ be a prime number, and $B$ the group
$\Z_p^\omega/\bigoplus_{n\in\omega} p^n\,\Z_p.$
Then $B$ is cotorsion, but fails to satisfy the conclusion
of Lemma~\ref{L.ab_divisible}; hence $B$ is not algebraically compact.
\end{lemma}

\begin{proof}
As a homomorphic image of an algebraically compact group, $B$
is cotorsion.

To see the failure of the conclusion of Lemma~\ref{L.ab_divisible},
let $x\in B$ be the image of $(p^n)_{n\in\omega}\in\Z_p^\omega.$
(Note that the above coordinates $p^n$ are ``ghosts'', in the
sense that any {\em finite} set of them may, by the definition of $B,$
be changed to $0$ without changing the element $x.)$
For each $n>0,$ if we let $x_n\in B$ be the image of the element
of $\Z_p^\omega$ whose coordinate in position $m$ is $0$ for $m<n,$
and $p^{m-n}$ for $m\geq n,$ then we see that $x=p^n\,x_n.$
Hence $x\in B'.$

Now let $y$ be any element satisfying $x=p\,y.$
Writing $y$ as the image of $(a_n)_{n\in\omega}\in\Z_p^\omega,$
we see from the definition of $B$
that for all but finitely many $n$ we must have $a_n=p^{n-1}.$
(And note that coordinates with this property are not ``ghosts''!)
But for any $n$ such that this relation holds,
we can see by looking at the $\!n\!$-th
coordinate that $y\notin p^n B.$
So $y\notin B';$ and since we have shown
this for all $y$ with $x=p\,y,$ we have $x\notin p\,B'.$
Since $x\in B',$ this shows that $B'\neq p\,B'.$
\end{proof}

(L.\,Fuchs (personal communication) points out
another way to see that the above group $B$ is not
algebraically compact:  by noting that its torsion subgroup
$\bigoplus_{n\in\omega} \Z_p/p^n\,\Z_p$ is not torsion-complete,
and calling on \cite[Theorem~68.4, (ii)$\!\implies\!$(i)]{Fuchs2}.)

Note that any group $B$ which, like the one constructed above,
is cotorsion but not algebraically compact is,
by the former fact, a homomorphic image of
a nonprincipal countable ultraproduct of groups, but by
Lemma~\ref{L.ab_ultra_retr}
\eqref{d.ab_oE_retr}$\!\implies\!$\eqref{d.ab_alg_cp}, does not
admit a left inverse to a diagonal embedding $B\to B^\omega/\U,$
confirming the assertion
made in the second paragraph after Question~\ref{Q.B,B',B''}.

Let us obtain, next, some restrictions on
the class of cotorsion abelian groups.
These will allow us to deduce that many sorts
of groups are not cotorsion, and so give more
examples to which we can apply Theorem~\ref{T.fin_many_i}.
In the next lemma we combine the fact that the
cotorsion groups are the homomorphic images of the algebraically
compact groups with another of the criteria for algebraic
compactness given in \cite[Theorem~38.1]{Fuchs1},
namely, that an abelian group $C$ is algebraically compact if
and only if it is {\em pure-injective}, meaning that for any
pure subgroup $A_0$ of an abelian group $A_1,$ every homomorphism
$A_0\to C$ extends to a homomorphism $A_1\to C.$
In an earlier version of this note, I asked whether the
direction~``\eqref{d.extend}$\!\implies\!$cotorsion'' in the lemma held;
I am indebted to K.\,M.\,Rangaswamy and Manfred Dugas for
(independently) showing me why it does.

\begin{lemma}\label{L.pure}
An abelian group $B$ is cotorsion if and only if it satisfies
\begin{equation}\begin{minipage}[c]{35pc}\label{d.extend}
For every abelian group $A$ having
a pure subgroup $F$ which is free abelian,
every homomorphism $F\to B$ extends to a homomorphism $A\to B.$
\end{minipage}\end{equation}
\end{lemma}

\begin{proof}
Assuming $B$ cotorsion, let us write it as a homomorphic image
of an algebraically compact abelian group $C.$
Since $F$ is free, we can lift the given map $F\to B$ to a map
$F\to C,$ and then, since $C$ is algebraically compact, equivalently,
pure-injective, we can extend that lifted map to a map $A\to C.$
Composing with our map $C\to B,$
we get the desired extension to $A$ of the given map $F\to B.$

Conversely, assuming~\eqref{d.extend},
write $B$ as a homomorphic image of a free abelian group $F.$
Now by \cite[\S38, Exercise~8, p.\,162]{Fuchs1}, every abelian
group embeds as a pure subgroup in a group admitting
a compact Hausdorff group topology; let $A$ be such an overgroup of $F.$
(For an explicit embedding in  this case,
let $\hat{\Z}$ denote the completion of $\Z$
with respect to its subgroup topology.
Then $\Z$ is a pure subgroup of the compact group $\hat{\Z},$
so writing $F=\bigoplus_I \Z,$ we see that
$F$ is pure in the compact group $\hat{\Z}^I.)$
By~\eqref{d.extend}, our homomorphism of $F$
onto $B$ extends to a homomorphism of $A$ onto $B,$ so
by~Proposition~\ref{P.cotor},
\eqref{d.ab_im_cp}$\!\implies\!$\eqref{d.cotorsion}, $B$ is cotorsion.
\end{proof}

Our first application of this result will show that in a
cotorsion abelian group $B,$ highly divisible elements abound;
for instance, that if $p_1$ and $p_2$ are distinct primes, then
every element of $B$ is the sum of an element divisible by all
powers of $p_1$ and an element divisible by all powers of $p_2.$
To state the result in greater generality, let us, for any set $P$
of primes, write $\Z[P^{-1}]$ for the subring of $\Q$ consisting of
elements whose denominators lie in the multiplicative monoid
generated by $P,$ and call an element $x$
of an abelian group~$A$ {\em $\!P\!$-divisible} if it lies
in the image of a homomorphism from the additive group of
$\Z[P^{-1}]$ to $A.$
We shall call an abelian group $\!P\!$-divisible
if all its elements are.

\begin{proposition}\label{P.Pj}
If $B$ is a cotorsion abelian group,
and $P_0,\dots,P_{m-1}$ are sets of prime numbers such that
$P_0\cap\dots\cap P_{m-1}=\emptyset,$
then every element $b\in B$ can be written
$b_0+\dots+b_{m-1},$ where for each $j,$ $b_j$ is $\!P_j\!$-divisible.
Equivalently, $B$ is a sum of subgroups $B_0+\dots+B_{m-1}$
such that each group $B_j$ is $\!P_j\!$-divisible.
\end{proposition}

\begin{proof}
Let $A$ be the additive group of
$\Z[P_0^{-1}]\times\dots\times\Z[P_{m-1}^{-1}],$
and $F$ the infinite cyclic subgroup thereof generated by $(1,\dots,1).$
That the inclusion $F\subseteq A$ is pure follows from
the fact that $P_0\cap\dots\cap P_{m-1}=\emptyset.$
Indeed, if an element $d(1,\dots,1)\in F$ is not divisible in $F$ by
some positive integer $n,$ then $d$ is not divisible by $n,$
so $n$ has a prime power factor $p^i$ not dividing $d.$
Choosing $k$ such that
$p\notin P_k,$ we see that the $\!k\!$-th coordinate of
$d(1,\dots,1)$ is not divisible by $p^i$ in $\Z[P_k^{-1}],$
so in $A,$ $d(1,\dots,1)$ is not divisible by $p^i,$
hence not divisible by $n.$

Hence by Lemma~\ref{L.pure}, for any $b\in B,$ the map $F\to B$ taking
$(1,\dots,1)$ to $b$ extends to $A,$
giving a representation of $b$ as the sum of the images of
the elements $(0,\dots,1,\dots,0),$ each of which is
$\!P_j\!$-divisible for some $j.$
The equivalence of this result to
the final statement of the lemma follows from the
fact that for any set $P$ of primes, the $\!P\!$-divisible elements of
an abelian group form a subgroup.
\end{proof}

As a quick illustration, consider the group $\Z_p$ of $\!p\!$-adic
integers, which we have seen is algebraically
compact, and hence cotorsion.
That group is $\!P\!$-divisible for $P$ the set of all primes
other than $p.$
Given $P_0,\dots,P_{m-1}$ as in Proposition~\ref{P.Pj}, at least
one $P_j$ will fail to contain $p,$ so $\Z_p$ is
$\!P_j\!$-divisible for that~$j,$ confirming the
conclusion of the proposition.

Of course, the much smaller group of rational numbers with
denominators relatively prime to $p$ (of which the group $\Z_p$
is a completion) is $\!P\!$-divisible for the
same set $P,$ and so also satisfies the conclusion of
Proposition~\ref{P.Pj}.
However, that group is not cotorsion.
Indeed, from the characterization of slender abelian groups recalled
immediately after the proof of Proposition~\ref{P.slender},
every abelian group which is torsion-free and which contains no
copy of $\Q$ and has less than continuum cardinality is slender,
hence, if nonzero, is non-cotorsion.

The next result generalizes the above restriction on cotorsion groups.

\begin{proposition}\label{P.ab_cap_dB}
If $B$ is a cotorsion abelian group such that
$dB\neq\{0\}$ for every positive integer $d,$
but $\bigcap_{d\in\Z,\,d>0} dB=\{0\},$
then $B$ has at least continuum cardinality.
\end{proposition}

\begin{proof}
We shall construct a homomorphism $\bigoplus_\omega\Z\to B,$
extend it to a map $\Z^\omega\to B$ by Lemma~\ref{L.pure}, and
show that under the extended map,
continuum many elements of $\Z^\omega$ have distinct images.
We begin by carefully selecting the elements to which to
send the free generators of $\bigoplus_\omega\Z.$

I claim that we can choose
positive integers $d_0,d_1,\dots\,,$ each a multiple of the
one before, and elements $b_0,b_1,\dots\in B,$
such that for each $n\in\omega,$ we have $d_n b_n\notin d_{n+1} B.$
We start with $d_0=1,$ and $b_0$ any nonzero element of $B.$
Assuming that for some $n\geq 0,$ $d_n$ and $b_n$ have been
chosen with $d_n b_n\neq 0,$
the hypothesis $\bigcap_{d\in\Z,\,d>0} dB=\{0\}$ allows us to choose
$d_{n+1}>0$ such that $d_n b_n\notin d_{n+1}B.$
Replacing $d_{n+1}$ by a proper multiple if necessary,
we may assume $d_n\,|\,d_{n+1}.$
Using the fact that $d_{n+1} B\neq\{0\},$ we can then choose
$b_{n+1}$ such that $d_{n+1} b_{n+1}\neq 0.$
Continuing recursively, we get
$d_0,d_1,\dots$ and $b_0,b_1,\dots$ with the asserted properties.

We now map $\bigoplus_\omega\Z$ to $B$
by sending each $\delta_n$ to $b_n.$
Since $\bigoplus_\omega\Z$ is
a pure subgroup of $\Z^\omega,$ Lemma~\ref{L.pure}
allows us to extend this map to a homomorphism
$f:\Z^\omega\to B,$ which still carries each $\delta_n$ to $b_n.$

For each $\varepsilon=(\varepsilon_n)_{n\in\omega}\in\{0,1\}^\omega,$
let $\varepsilon d$ denote
$(\varepsilon_0 d_0,\dots,\varepsilon_n d_n,\dots)\in\Z^\omega.$
I claim that distinct strings $\varepsilon$ yield distinct
elements $f(\varepsilon d)\in B.$
Indeed, for $\varepsilon\neq\varepsilon',$
let $n\in\omega$ be the least index such that
$\varepsilon_n\neq\varepsilon'_n,$ and
let us write $f(\varepsilon d)$ as
$f(\varepsilon_0 d_0,\dots,\varepsilon_n d_n,0,0,\dots)+
f(0,\dots,0,\varepsilon_{n+1}d_{n+1},\,\varepsilon_{n+2}d_{n+2},\dots).$
If we compare this with the corresponding expression for
$f(\varepsilon' d),$ we see that the left-hand summands in these
expressions differ by exactly $f(d_n\delta_n),$ i.e.,
$d_n b_n,$ which by assumption does
not lie in $d_{n+1} B;$ while the right-hand
summands {\em do} lie in $d_{n+1} B,$ since for all $m\geq n$
we have $d_{n+1}|d_m.$
Hence $f(\varepsilon d)-f(\varepsilon' d)\neq 0;$
so we indeed have continuum many distinct elements of~$B.$
\end{proof}

As an application, it is easy to deduce that no subgroup $B$ of
$\prod_{\mbox{\scriptsize{primes }}p}\,\Z/p\Z$ which is
infinite, but of less than continuum cardinality, can be cotorsion.
Hence, if we take such a subgroup with no $\!2\!$-torsion,
containing an element $b$ of infinite order,
its semidirect product with $\pm 1$ will again be a group to which
Theorem~\ref{T.fin_many_i} applies.

On the other hand, we saw in Lemma~\ref{L.TS} that for
the semidirect product of $\{\pm 1\}$ with the group $\Q,$
the conclusion of Theorem~\ref{T.fin_many_i} fails; and
Proposition~\ref{P.compact} shows the same for the semidirect
product of $\{\pm 1\}$ with any {\em finite} abelian group.
In fact, $\Q$ and all finite abelian groups are cotorsion;
the next result includes these statements as special cases.
It is curious that its formulation is analogous to that of
Proposition~\ref{P.compact}, but the reasoning is quite different.

\begin{proposition}[cf.~{\cite[p.178, last paragraph of Notes]{Fuchs1}}]\label{P.injective}
Let $B$ be an abelian group which is divisible, or is of
finite exponent, or more generally, is
the sum of a divisible group and one of finite exponent;
or, still more generally, is the underlying
additive group of an injective module over some ring $R.$
Then for any set $I$ and any ultrafilter $\U$ on $I,$ there
is a group homomorphism $B^I/\,\U\to B$ left-inverse to the
natural embedding $B\to B^I\to B^I/\,\U.$

Hence by Lemma~\ref{L.ab_ultra_retr},
\eqref{d.ab_oE_retr}$\!\implies\!$\eqref{d.ab_alg_cp}, every
such $B$ is algebraically compact, and so in particular, is cotorsion.
\end{proposition}

\begin{proof}
First suppose $B$ has the property introduced above by the
words ``still more generally''.
Then the maps $B\to B^I\to B^I/\,\U$ are $\!R\!$-module
homomorphisms whose composite is an embedding.
The injectivity of $B$ as an $\!R\!$-module
thus yields the desired left inverse map.
Taking $I=\omega$ and $\U$ nonprincipal, we conclude that $B$ is
algebraically compact
(by Lemma~\ref{L.ab_ultra_retr},
\eqref{d.ab_oE_retr}$\!\implies\!$\eqref{d.ab_alg_cp}).

It remains to show that the various sorts of abelian groups
named are indeed injective modules over appropriate rings.
Any divisible abelian group is an injective $\!\Z\!$-module by
\cite[Proposition~3.19]{TYL}.
An abelian group $B$ of finite exponent $n$ can be
written as a direct product
of free $\!\Z/d\Z\!$-modules as $d$ ranges over the divisors
of $n;$ and each of the rings $\Z/d\Z$ is self-injective,
so that its free modules are injective by
\cite[Corollary~3.13(1) and Theorem~3.46(4)$\!\implies\!$(2)]{TYL}.
Finally, if $B$ is the sum of a divisible subgroup $D$ and a
subgroup $E$ of finite exponent, then the injectivity of $D$
over $\Z$ allows us to split it off as a direct summand,
and the complementary summand will be a homomorphic image $E'$ of $E,$
hence again of finite exponent.
We can now make $B=D\oplus E'$ a module over
the direct product $R$ of $\Z$ and finitely many rings $\Z/d\Z,$
in such a way that the component over each of these factor rings
is injective over that ring.
The group $B$ will then be injective over~$R.$
\end{proof}

I do not know the answer to

\begin{question}\label{Q.ab_<=}
For an abelian group $B$ to be cotorsion, is it sufficient
that every homomorphism $\bigoplus_\omega\Z\to B$
extend to a homomorphism $\Z^\omega\to B$?
\textup{(}In other words, in Lemma~\ref{L.pure},
is condition~\eqref{d.extend}
equivalent to the special case where the inclusion $F\subseteq A$
is $\bigoplus_\omega\Z\subseteq\Z^\omega$?\textup{)}
\end{question}

The following example shows that the converse of
Corollary~\ref{C.slender_vs_cotorsion} is not true: a group
$B$ with no nonzero slender homomorphic image need not be cotorsion.

\begin{lemma}\label{L.B/T(B)}
Within the group $A=\prod_{\mbox{\scriptsize\rm{primes }}p}\,\Z/p\Z,$
let $u$ be the element having $1$ in every coordinate,
and let $B$ consist of all elements $b\in A$ such that
$db=nu$ for some integer $n$ and nonzero integer $d$
\textup{(}mnemonic for ``numerator'' and ``denominator''\textup{)}.

Then $B$ is a countable subgroup of $A,$ such that every
cotorsion subgroup of $B$ is torsion \textup{(}so that $B$
is {\em not} itself cotorsion\textup{)}, but the factor-group
of $B$ by its torsion subgroup is isomorphic to $\Q,$ and so
is cotorsion.

Hence $B$ has no nonzero slender homomorphic images.
\end{lemma}

\begin{proof}
That $B$ is a subgroup of $A$ is immediate.
It is countable because each $b\in B$ is determined by any choice
of $n$ and $d$ satisfying $db=nu,$ together with the
coordinates of $b$ at the finitely many primes dividing $d.$

By the observation following the proof of Proposition~\ref{P.ab_cap_dB},
cotorsion subgroups of $B$ are finite, hence are torsion.

On the other hand, the factor group of $B$ by its torsion
subgroup is isomorphic to $\Q$
via the map sending the image of each $b\in B$ to the common value
of $n/d\in\Q$ for all relations $db=nu$ satisfied by $b;$ and
$\Q,$ being divisible, is cotorsion by Proposition~\ref{P.injective}.

Since slender groups are torsion-free, a homomorphism $f$
from $B$ to a slender group must annihilate the torsion subgroup
of $B,$ hence $f(B)$ must be a homomorphic image of $\Q,$ hence
by Corollary~\ref{C.slender_vs_cotorsion} must be zero.
\end{proof}

Here is a question of a different flavor.

\begin{question}\label{Q.ab<gp}
If an abelian group $B$ can be written as a homomorphic image of a
nonprincipal countable ultraproduct of
{\em not necessarily abelian} groups $G_n,$ must it be
a homomorphic image of a nonprincipal countable ultraproduct of
{\em abelian} groups, i.e., must it be cotorsion?
\end{question}

The reason this question is nontrivial is that
abelianization does not commute with ultraproducts.
For instance, let $G$ be a group which is perfect
(satisfies $G=[G,\,G])$ but which
for each $n$ has an element $x_n$ that cannot be
written as the product of fewer than $n$ commutators.
(The latter property is called
``infinite commutator width'';
for examples of such $G$ see \cite{width}.)
Then no nonprincipal ultrapower $G^\omega/\U$ will be perfect,
because for such a family of elements $x_n,$
the image of $(x_n)_{n\in\omega}\in G^\omega$ in
$G^\omega/\U$ will not be a product of finitely many commutators.
Hence the abelianization $B$ of $\prod_{n\in\omega} G_n/\U$ is
a nontrivial abelian group
satisfying the hypothesis of Question~\ref{Q.ab<gp}, but there is
no obvious candidate for a representation of $B$ as in the conclusion
of that question.

We can, however, prove a weak result in the direction
of a positive answer.

\begin{lemma}\label{L.ab<gp}
If an abelian group $B$ can be written as a homomorphic image of a
nonprincipal countable ultraproduct $\prod_{n\in\omega} G_n/\U$ of
not necessarily abelian groups, then $B$ is
a directed union of cotorsion abelian subgroups.
\end{lemma}

\begin{proof}
By Lemma~\ref{L.ultra}, every cyclic subgroup of $B$ is
contained in a cotorsion subgroup.
Now the class of cotorsion
abelian groups, as characterized by any of~\eqref{d.ab_im_redpr},
\eqref{d.ab_im_frechet} or~\eqref{d.ab_im_cp}, is easily seen to be
closed under finite direct sums, hence since it is closed under
homomorphic images, it is closed under finite sums
in abelian overgroups, so the cotorsion subgroups
of such an overgroup form a directed system.
\end{proof}

But not every directed union of cotorsion groups is cotorsion.
For instance, every torsion abelian group is the
directed union of its finite subgroups, which are cotorsion
by Proposition~\ref{P.injective};
but by Proposition~\ref{P.ab_cap_dB}, the group
$\bigoplus_{\mbox{\scriptsize\rm{primes }}p}\,\Z/p\Z$
is not cotorsion.

We remark that Questions~\ref{Q.ultra} and~\ref{Q.ab<gp}
cannot both have positive answers,
since as noted earlier, a positive answer to Question~\ref{Q.ultra}
would make every torsion abelian
group, including the abovementioned group
$\bigoplus_{\mbox{\scriptsize\rm{primes }}p}\,\Z/p\Z,$ a homomorphic
image of a countable ultraproduct of (not necessarily abelian) groups.
But as we also said earlier, a positive answer to
Question~\ref{Q.ultra} seems highly unlikely.

A noticeable difference between our results on
general groups in \S\S\ref{S.gp_ultra}-\ref{S.gp_prin}, and our
results on abelian groups in the above three sections, is that in the
former, we composed maps $\prod_{i\in I} G_i\to B$ with
the natural map $B\to B/Z(B)$ before looking at factorization
properties, but we have done nothing of the sort for abelian groups.
It might be of interest to see whether one can
improve the results of these sections by composing
homomorphisms $\prod_{i\in I} A_i\to B$ with the map $B\to B/X(B)$
for some natural choice of $X(B),$ such as the torsion subgroup
of $B,$ the subgroup of divisible elements, their sum, or
the sum of all cotorsion subgroups of $B.$
(Lemma~\ref{L.B/T(B)} shows that for the last of these
choices, $B/X(B)$ may not itself be cotorsion-free; but this
need not be a problem; cf.\ the fact that
for a nonabelian group $B,$ the group $B/Z(B)$ may not
have trivial center.)
In the opposite direction, it might be possible to strengthen the
results of \S\S\ref{S.gp_ultra}-\ref{S.gp_prin} by dividing
$B,$ not by $Z(B),$ but by a smaller subgroup $X(Z(B))$
for one of the above constructions $X.$
I leave these ideas for others to explore.

\section{Some related questions that have been studied}\label{S.ab_Chase&&}

The direct sum $\bigoplus_{i\in I} A_i$ of a family of abelian
groups -- or more generally, of a family of modules over any ring
$R$ -- is their {\em coproduct} in the category of abelian groups
or $\!R\!$-modules; hence for such objects,
their coproduct can be regarded as
the subgroup or submodule of elements of finite support in their
direct product $\prod_{i\in I} A_i.$
Now in any category, a homomorphism from a coproduct
of objects $A_i$ to an arbitrary object $B$ is determined simply by
choosing a homomorphism from each $A_i$ to $B.$
So the phenomena we have been investigating in the last two sections
can be looked at as consequences
of the fact that not every such map on a coproduct of abelian
groups can be extended consistently to the
elements of $\prod_{i\in I} A_i$ with infinite supports.
The slender modules are those modules $B$ for which this
restriction on maps to $B$ is so strong that it can only be
satisfied by maps that factor
through the product of finitely many of the $A_i.$

Dually, one gets homomorphisms from an abelian group
or $\!R\!$-module $B$ to a direct product $\prod_{j\in J} C_j$ simply
by choosing a homomorphism into each
$C_j;$ but if we wish to map $B$ into the coproduct
$\bigoplus_{j\in J} C_j\subseteq\prod_{j\in J} C_j,$ we face the
problem of choosing those homomorphisms so that the resulting
map takes each element of $B$ to an element of finite support.
The question of which modules $B$ have the property that the
only way to achieve this is by mapping
into a finite subsum of $\bigoplus_{j\in J} C_j$
is answered by El Bashir, Kepka and N\v{e}mec in
Proposition~4.1 of \cite{EKN}; that paper also studies the
corresponding questions for colimit constructions other than coproducts.

Several workers, beginning with
Chase \cite{Chase1}, \cite{Chase2}, have looked at
the two-headed situation of
module homomorphisms $f:\prod_{i\in I} A_i\to\bigoplus_{j\in J} C_j.$
Here one may ask when every such map is a sum of one homomorphism
which factors through the projection of $\prod_{i\in I} A_i$ onto
a finite subproduct, and another which factors through the inclusion of
a finite subsum in $\bigoplus_{j\in J} C_j.$
Just as, in studying nonabelian groups
in~\S\S\ref{S.gp_ultra}-\ref{S.gp_prin}, we found it desirable
to divide out by $Z(B)$ to avoid certain easy ways that maps
could involve infinitely many factors, so in the results
of this sort, two adjustments turn out to be useful: dividing out
by submodules of ``highly divisible'' elements of the $C_j,$
and multiplying the given homomorphism by some nonzero ring
element $d;$ which essentially means restricting it to
$\prod_{i\in I} d A_i.$
Thus, \cite[Theorem~1.2]{Chase2}, more or less the
starting point for the development of the subject, says,
if restricted to the case where the base ring is $\Z$ and
where a certain filter of principal right ideals
in the statement of that
theorem consists of {\em all} the nonzero ideals of $\Z,$
that given any homomorphism of abelian groups
$f:\prod_{n\in\omega} A_n\to\bigoplus_{i\in I} C_i,$
there exists an integer $d>0$ such
that when $f$ is applied to $\prod_{n\in\omega} d A_n,$
and followed by the factor map
$\bigoplus_{i\in I} C_i\to\bigoplus_{i\in I} (C_i/\bigcap_{e>0} eC_i),$
it carries the product of some cofinite subfamily of the $d A_n$
into the sum of a finite subfamily of the $C_i/\bigcap_{e>0} eC_i.$

That result (in its general module-theoretic form)
is strengthened in \cite[Theorem~2]{MD+BZB} to
allow products $\prod_{i\in I} A_i$ over any index set $I$ of
cardinality
less than all uncountable measurable cardinals, and to remove the
requirement that the right ideals considered be principal, while
in \cite[Theorem~9]{P_vs_cP}, the direct product is replaced by
a general inverse limit.
For further related work, see references
in \cite[first paragraph of p.46]{P_vs_cP}.

(It is curious that the proof of the abovementioned theorem from
\cite{Chase2}, and that
of Proposition~\ref{P.ab_cap_dB} of this note
use virtually the same construction, but for
very different purposes: in \cite{Chase2}, to obtain a
contradiction by constructing an element whose image
in the direct sum would have infinitely many nonzero components;
in Proposition~\ref{P.ab_cap_dB}, to get
continuum many distinct elements in the image of our map.)

Turning back to the results of the three preceding sections, it would,
of course, be desirable to investigate the corresponding
questions with abelian groups replaced
by modules over a general ring $R.$
In \cite[Chapters~7--8]{J+L}, algebraically compact modules over
general $R$ are studied, but cotorsion modules are not mentioned.
(There {\em are} numerous MathSciNet results for ``cotorsion module'',
but I have not had time to examine them.)
One might also take a hint from \cite{Chase1}, \cite{Chase2},
and see whether
one gets nonobvious variants of our results if one considers those
$B$ such that all homomorphisms $\prod_{i\in I} A_i\to B$ acquire
the factorization properties we are looking for {\em after} multiplying
$\prod_{i\in I} A_i$ by some integer (or ideal), and/or
dividing $B$ by an appropriate subgroup (or submodule) of highly
divisible elements.
(This is related to the suggestion in the last paragraph of the
preceding section.)

\section{Rings.}\label{S.rings}

As mentioned in \S\ref{S.intro}, the results of this note were inspired
by investigations of factorization properties of
homomorphisms on direct products of not-necessarily-associative
algebras over an infinite field \cite{prod_Lie1},
\cite{prod_Lie2}, \cite{cap_ultra}, \cite{FM}.
In those papers, the assumption of infinite base field
was used to show that, under appropriate bounds on the
vector-space dimension of $B,$
the ultrafilters occurring had to be principal.

If we look at not-necessarily-associative {\em rings},
without assuming a structure of algebra over a field,
then as we shall see below, we can still get results
analogous to the main results of \S\ref{S.gp_ultra}
(on when maps must factor through finitely many ultraproducts)
and those of \S\ref{S.gp_prin} (saying that such ultraproducts
must be principal under appropriate assumptions on
the additive structure of $B).$
Between these we shall insert Proposition~\ref{P.ring_main},
which will say that if our rings
have unit, then the absence of factorization through finitely many
ultraproducts implies the existence of an
associative commutative subring of $B$ with
the cardinality of the continuum, of an explicitly describable form,
over which $B$ becomes an algebra.
I will not repeat here the results on algebras over an infinite field
from the papers cited above; and having
spent many words on those papers, I will be brief in this section.

In a direct product ring $\prod_{i\in I}R_i,$
we define the support of an element $x=(x_i)_{i\in I}$
to be $\{i\in I\mid x_i\neq 0\}.$
Whereas in \S\S\ref{S.gp_ultra}-\ref{S.gp_prin}, our basic tool
was the {\em commutativity} of elements with disjoint supports in
a product group, and the phenomenon that this tool could not
handle was avoided by dividing out by the
center, $Z(B),$ the corresponding
tool in \cite{prod_Lie1}, \cite{prod_Lie2}, \cite{cap_ultra}, \cite{FM}
was the fact that ring elements with disjoint supports
have product zero; and the ideal one had to divide out by
(which was also denoted $Z(B))$ was the zero-multiplication ideal.
In this section, for $B$ a ring, we will, as in those papers, write
\begin{equation}\begin{minipage}[c]{35pc}\label{d.ring_B}
$Z(B)\ =\ \{b\in B\mid b\,B=B\,b=0\}.$
\end{minipage}\end{equation}
As in \S\S\ref{S.gp_ultra}-\ref{S.gp_prin},
we let $\pi: B\to B/Z(B)$ be the quotient homomorphism.

There was one commutativity result in
Theorem~\ref{T.gp_main} above that arose for a reason other than that
elements in different factors of a direct product commute,
namely,~\eqref{d.commute}, which followed from the fact that every
element commutes with itself.
Thus, the analog of that one statement,~\eqref{d.ring_commute}
below, again concerns commutativity, rather than zero products.

The obvious analog of Lemma~\ref{L.gp_Us} holds for rings,
and yields the following analog of Theorem~\ref{T.gp_main}.

\begin{theorem}\label{T.ring_main}
Let $B$ be a ring \textup{(}understood here to mean an abelian
group given with an arbitrary bilinear multiplication $B\times B\to B),$
and suppose there exist a family $(R_i)_{i\in I}$
of rings, and a surjective
ring homomorphism $f:\prod_{i\in I} R_i\to B,$ such that
the induced homomorphism $\pi f:\prod_{i\in I} R_i\to B/Z(B)$
does not factor
through the natural map from $\prod_{i\in I} R_i$ to any finite
product of ultraproducts of the $R_i.$
Then $B$ contains families of elements $(a_S),$ $(b_S),$
indexed by the subsets $S\subseteq\omega,$ such that
\begin{equation}\begin{minipage}[c]{35pc}\label{d.ring_commute}
All the elements $a_S$ $(S\subseteq\omega)$ commute with one another,
and all the elements $b_S$ likewise commute with one another.
\end{minipage}\end{equation}
\begin{equation}\begin{minipage}[c]{35pc}\label{d.ring_a_disj}
For {\em disjoint} subsets $S$ and $T$ of $\omega,$
we have $a_S + a_T=a_{S\cup T},$ $b_S + b_T=b_{S\cup T},$
and $0=a_S\,a_T=a_S\,b_T=b_S\,a_T=b_S\,b_T.$
\end{minipage}\end{equation}
\begin{equation}\begin{minipage}[c]{35pc}\label{d.ring_one}
For subsets $S$ and $T$ of $\omega$ with $\card(S\cap T)=1,$
we have $a_S\,b_T\neq 0.$\qed 
\end{minipage}\end{equation}
\end{theorem}

One gets from this the obvious analog of Corollary~\ref{C.c},
which I will not write down, only noting one minor way in which
the statement is weaker than that corollary:
In a nonassociative ring,
a family of pairwise commuting elements need not generate a commutative
subring, so the assertion of commutativity in the
last sentence of Corollary~\ref{C.c} disappears here.

One likewise has the analog of Theorem~\ref{T.gp_CC}.
Namely, following \cite[Definitions~13 and~15]{prod_Lie1},
we define an almost direct decomposition of a ring $B$ as
an expression $B=B'+B'',$ where $B'$ and $B''$ are ideals of $B,$
each of which is the $\!2\!$-sided annihilator of the other;
and we shall say that $B$ has chain condition on almost direct
factors if every chain of such ideals is finite.
Then we get

\begin{theorem}[cf.\ {\cite[Proposition~16]{prod_Lie1}}]\label{T.ring_CC}
Let $B$ again be a ring such that there exist a family of rings
$(R_i)_{i\in I},$ and a surjective ring homomorphism
$f:\prod_{i\in I} R_i\to B$ such that the induced homomorphism
$\pi f:\prod_{i\in I} R_i\to B/Z(B)$ does not factor
through the natural map of $\prod_{i\in I} R_i$ to any finite
direct product of ultraproducts of the $R_i.$

Then $B$ does not have chain condition on almost direct factors.
In fact, it has a family of almost direct factors order-isomorphic
to the lattice $2^\omega,$ and forming a sublattice
of the lattice of ideals of $B.$
\end{theorem}

So far we have not assumed our rings unital, since that hypothesis
is unnatural for many important classes of nonassociative rings.
The next result shows how in the
unital case, the above theorems can be simplified and strengthened.
For unital rings $B$ we have $Z(B)=0,$ so
$B/Z(B)$ everywhere becomes $B.$
Moreover, we can take each of the systems of elements
$x_n,\,y_n\in\prod_{i\in J_n} R_i$ from which we obtain
the elements $a_S$ and $b_S$ in Theorem~\ref{T.ring_main} to
consist of the multiplicative identity elements of the rings
$\prod_{i\in J_n} R_i.$
With a little further work, we shall get:

\begin{proposition}\label{P.ring_main}
Under the common hypotheses of Theorem~\ref{T.ring_main}
and~\ref{T.ring_CC}, if the rings
$B$ and $R_i$ are unital, with homomorphisms preserving
multiplicative identity elements, then $B$ is a faithful unital algebra
over a commutative associative unital subring
of the form $\prod_{n\in\omega} \Z/d_n\Z,$
where each $d_n$ is a nonnegative integer $\neq 1.$
Moreover, one can take all but one of the $d_n$ to be
equal, and that one to be a multiple \textup{(}not necessarily
proper\textup{)} of the common value of the others.
\end{proposition}

\begin{proof}[Sketch of proof]
As in the proof of Theorem~\ref{T.gp_main}, the non-factorization
of $f$ tells us that we can partition $I$ into subsets
$J_n$ $(n\in\omega)$ such that for
each $n,$ $f(\prod_{i\in J_n} R_i)\neq\{0\}.$
(Here we regard the direct product of each subfamily
of the $R_i$ as an ideal of $\prod_{i\in I} R_i.$
Each of these ideals has a multiplicative identity element, generally
different from that of $\prod_{i\in I} R_i.)$
For each $S\subseteq\omega,$ let $x_S$ denote the multiplicative
identity element
of $\prod_{i\in\bigcup_{n\in S} J_n} R_i\subseteq\prod_{i\in I} R_i,$
and let $a_S=f(x_S).$
We see that the operations
of multiplication by $x_S$ and $x_{\omega-S}$ are idempotent
endomorphisms of the additive group
of $\prod_{i\in I} R_i,$ which give the projection
homomorphisms to the two factors of its decomposition as
$(\prod_{i\in \bigcup_{n\in S} J_n} R_i)\times
(\prod_{i\in \bigcup_{n\in\omega-S} J_n} R_i).$
Hence their images $a_S$ and $a_{\omega-S}$ likewise determine a direct
product decomposition of the ring $B.$

Now for every $S\subseteq\omega,$ let $c_S$ denote the nonnegative
integer such that the additive subgroup of $B$ generated
by $a_S$ is isomorphic to $\Z/c_S\Z$
(the characteristic of
the ring $f(\prod_{i\in \bigcup_{n\in S} J_n} R_i)).$
Note that for $\emptyset\neq T\subseteq S,$
we have $1\neq c_T\,|\,c_S.$

The behavior of $c_S$ as a function of $S$ can be complicated;
but with the help of the Noetherian property of the integers,
we can find a family of subsets of $\omega$ on which that
function has an easy description.
Namely, let us choose, among all {\em infinite} $S\subseteq\omega,$
one which gives a maximal value for the ideal $c_S\Z.$
Then for every infinite subset $T\subseteq S,$
we necessarily have $c_T=c_S.$
Hence, let us partition $S$ into countably many infinite
subsets, $T_0,\dots, T_m,\dots\,,$ and use these
to partition $\omega$ into subsets $S_m,$ where for $m>0,$
we let $S_m=T_m,$ while we let $S_0=(\omega-S)\cup T_0.$
Thus, for $m>0$ we have $c_{S_m}=c_S$ by choice of $S,$
while $c_{S_0}=\r{lcm}(c_{\omega-S},\,c_{T_0}),$
a multiple of $c_{T_0}=c_S.$

Let us now map the ring $\Z^\omega$ into $\prod_{i\in I} R_i$
by sending each element $(e_m)_{m\in\omega}$ to the element whose
value on each factor
$\prod_{n\in S_m}(\prod_{i\in J_n} R_i)$ is $e_m$ times the
multiplicative identity element,
and then apply the map $f:\prod_{i\in I} R_i\to B.$
I claim that the image of $\Z^\omega$ in $B$ will be
isomorphic to $\prod_{m\in\omega}\Z/d_m\Z,$ where $d_m=c_{S_m}.$
Indeed, it is easy to verify
that an element of $\Z^\omega$ that goes to zero under the
componentwise map into $\prod_{m\in\omega}\Z/d_m\Z$
goes to zero in $B,$ as a result of our choice of $S$ and the $S_m.$
Conversely, if an
element $(e_m)_{m\in\omega}\in\Z^\omega$ does not have zero image
in $\prod_{m\in\omega}\Z/d_m\Z,$ we can choose an $m_0$ such that
$e_{m_0}$ is not divisible by $d_{m_0};$ and taking the image, in
$B,$ of the ring relation
$(e_m)_{m\in\omega}\,\delta_{m_0}=e_{m_0}\,\delta_{m_0}$
in $\Z^\omega,$
we see that the image of $(e_m)_{m\in\omega}$ in $B$ is also nonzero.

Finally, the fact that every ring $R$ is a $\!\Z\!$-algebra, and
that if $R$ has a multiplicative
identity element $1_R,$ its $\!\Z\!$-algebra
structure is induced by the operations of multiplication
in $R$ by members of $\Z\cdot 1_R,$ easily leads to the result that
$\prod_{n\in S_m}(\prod_{i\in J_n} R_i)$ is a $\!\Z^\omega\!$-algebra,
and that this structure leads to a structure of
$\!\prod_{m\in\omega}\Z/d_m\Z\!$-algebra on its homomorphic image $B.$
\end{proof}

Going back to not-necessarily-unital rings, and
turning to the question of when finitely many ultraproducts
through which a map factors must all be principal,
we can combine Theorem~\ref{T.not_prin} with the idea of
Lemma~\ref{L.ultra} to get the following result.

\begin{theorem}\label{T.principal}
Suppose $B$ is a ring which admits a surjective homomorphism
from a direct product ring, $f:\prod_{i\in I} R_i\to B,$
such that the composite $\pi f:\prod_{i\in I} R_i\to B\to B/Z(B)$
factors through the product of finitely many
ultraproducts of the $R_i,$ but not through the product
of finitely many countably complete ultraproducts.
\textup{(}So if $\card(I)$ is less than all uncountable
measure cardinals, if any exist, the latter condition simply says that
$\pi f$ does not factor through any finite product of the $A_i.)$

Then the additive group of $B/Z(B)$ has a nonzero cotorsion subgroup;
equivalently, it either contains nonzero elements of finite order, or
a copy of the additive group of $\Q,$ or a copy of the additive
group of the $\!p\!$-adic integers for some prime $p.$\qed
\end{theorem}

\section{Monoids}\label{S.monoids}

In studying homomorphisms from direct product monoids
onto a monoid $B,$
it is useful to assume some cancellation condition on $B.$
One that will suffice for our present purposes is
\begin{equation}\begin{minipage}[c]{35pc}\label{d.cancel}
$xy=x\implies y=e$ for $x,y\in B.$
\end{minipage}\end{equation}
Note that~\eqref{d.cancel} implies that one-sided inverses
are two-sided, since if $xy=e,$ we get $xyx=x,$ which
by~\eqref{d.cancel} gives $yx=e.$

We shall encounter two sorts of obstruction to
mapping infinite products onto $B$ in ways that indiscriminately
merge the factors.
On the one hand, there is the same effect of noncommutativity
that we took advantage of in the case of groups.
On the other hand, noninvertible elements create restrictions.
For instance, though linear algebra shows that the additive
group $\Q$ admits homomorphisms from the additive
group $\Q^\omega$ that behave arbitrarily on
$\bigoplus_\omega \Q,$ it is not hard to show that,
writing $\Q^{\geq 0}$ for the additive
monoid of nonnegative rational numbers, it is impossible
to have a homomorphism
$(\Q^{\geq 0})^\omega\to\Q^{\geq 0}$
that acts in a nonzero way on infinitely many of the
summands of $\bigoplus_\omega \Q^{\geq 0}.$

In our factorization results for groups, we
divided out by the center of $B;$ in the case of monoids, we will
divide out by the group of {\em central invertible} elements.
There are two versions of this concept: $Z(U(B)),$
the center of the group of units (invertible elements) of $B,$
and $U(Z(B)),$ the group of units of the center;
the former may be larger than the latter.
It is $U(Z(B)),$ the smaller of the two, that we will divide out by
(though the other will make a brief appearance in a proof).
Note that since $U(Z(B))$ consists of $\!B\!$-centralizing
invertible elements, one can speak (without distinguishing
right from left) of the orbits of $B$ under
multiplication by that group, and the set of such
orbits forms a factor-monoid $B/U(Z(B)).$
Clearly, the noninvertible elements of this factor-monoid
are precisely the cosets of the noninvertible elements of $B.$
We shall write $\pi$ for the projection map $B\to B/U(Z(B)).$

Given a monoid homomorphism $f:\prod_{i\in I} M_i\to B,$ we
define the analog of the filter $\F$ of~\eqref{d.F}, namely
\begin{equation}\begin{minipage}[c]{35pc}\label{d.monoid_F}
$\F\ =\ \{S\subseteq I\mid$ the composite map
$\prod_{i\in I} M_i\to B \to B/U(Z(B))$\\
\hspace*{13em} factors through the
projection $\prod_{i\in I} M_i\to \prod_{i\in S}M_i\}$\\[.2em]
\hspace*{1.5em}%
$=\ \{S\subseteq I\mid f(\prod_{i\in I-S} M_i)\subseteq U(Z(B))\}.$
\end{minipage}\end{equation}

The version of Lemma~\ref{L.gp_Us} that we will use for monoids
is not, as for rings, a carbon
copy of that lemma, so we shall give the statement and proof.
(But we will cut corners where the method of proof is the
same; so the reader might want to review the proof of
Lemma~\ref{L.gp_Us} before beginning this one.)
We do not yet assume the cancellativity condition~\eqref{d.cancel}.

\begin{lemma}\label{L.monoid_Us}
Let $f:\prod_{i\in I} M_i\to B$ be a homomorphism from a direct
product of monoids $M_i$ to a monoid $B,$ which is surjective
\textup{(}or more generally, such that the homomorphism
$\pi f:\prod_{i\in I} M_i\to B\to B/U(Z(B))$ is surjective\textup{)}.
Then the following two conditions are equivalent.
\begin{equation}\begin{minipage}[c]{35pc}\label{d.monoid_no_Us}
The homomorphism $\pi f:\prod_{i\in I} M_i\to B\to B/U(Z(B))$
does not factor through the natural map
$(\prod_{i\in I} M_i)/\U_0\times\dots
\times(\prod_{i\in I} M_i)/\U_{n-1}$ for any finite family
of ultrafilters $\U_0,\dots,\U_{n-1}$ on $I.$
\end{minipage}\end{equation}
\vspace{.2em}
\begin{equation}\begin{minipage}[c]{35pc}\label{d.monoid_Jn}
There exists a partition of $I$ into countably
many subsets $J_0,\,J_1,\dots\,,$ such that either\\[.2em]
\textup{(\ref{d.monoid_Jn}a)}\  Each submonoid
$\prod_{i\in J_n} M_i\subseteq \prod_{i\in I} M_i$ contains a
pair of invertible elements $x_n,$ $y_n$
whose images under $f$ do not commute in $B,$\\[.2em]
or\\[.2em]
\textup{(\ref{d.monoid_Jn}b)}\  Each submonoid
$\prod_{i\in J_n} M_i\subseteq \prod_{i\in I} M_i$ contains an
element $z_n$ whose image in $B$ is noninvertible.
\end{minipage}\end{equation}
\end{lemma}

\begin{proof}
To get~\eqref{d.monoid_Jn}$\!\implies\!$\eqref{d.monoid_no_Us},
note that in the situation of~(\ref{d.monoid_Jn}a), since for each
$n,$ $f(x_n)$ and $f(y_n)$ are invertible
and do not commute in $B,$ they do not lie in $Z(U(B)).$
Hence in particular $f(x_n)\notin U(Z(B)),$ so
by the final line of~\eqref{d.monoid_F}, $I-J_n\notin\F.$
That this implies~\eqref{d.monoid_no_Us} is seen as in
Lemma~\ref{L.gp_Us}.

If, rather, we are in the situation of~(\ref{d.monoid_Jn}b),
then the fact that the $f(z_n)$ are nonunits
implies that they do not lie in $U(Z(B)),$ giving the same result
for the same reason.

The proof of the converse begins, as for Lemma~\ref{L.gp_Us},
with the observation that~\eqref{d.monoid_no_Us}
implies that there exists a partition of $I$ into countably
many subsets $J_0,\,J_1,\dots,$ such
that each $\prod_{i\in J_n} M_i$ contains elements
mapped by $f$ to elements of $B$ not in $U(Z(B)).$
Let $L_n=\prod_{i\in J_n} M_i,$
so that $\prod_{i\in I} M_i=\prod_{n\in\omega} L_n.$
Clearly, it will either be true that for infinitely many $n,$ the
submonoid $f(L_n)\subseteq B$ contains a noninvertible element, or
that for infinitely many $n,$ that submonoid consists entirely of
invertible elements.

In the former case, those $J_n$ such that
$f(L_n)$ contains a noninvertible element of $B$
will constitute a partition of some subset $J\subseteq I$ into
countably many subsets.
If we enlarge one of these sets by throwing in the complementary
set $I-J,$ we get a partition of $I$ of the sort described
in~(\ref{d.monoid_Jn}b).

If, on the other hand, there are infinitely many $n$ such that
$f(L_n)$ consists entirely of invertible
elements of $B,$ then for each such $n,$ let us choose an
$x_n\in L_n$ such that $f(x_n)\notin U(Z(B)),$ and then a
$y\in\prod_{m\in\omega} L_m$ such that $f(y)$
does not commute with $f(x_n).$
As in the proof of Lemma~\ref{L.gp_Us}, we can
obtain from $y$ an element $y_n\in L_n$ such that
$f(y_n)$ still does not commute with $f(x_n).$
By assumption, $f(L_n)$ consists of invertible
elements, so $f(x_n)$ and $f(y_n)$ belong to $U(B).$
Thus, we have a partition of some $J\subseteq I$ into countably
many subsets as in~(\ref{d.monoid_Jn}a).
Again tacking $I-J$ onto one
of these, we can take this to be a partition of the whole set $I.$
\end{proof}

This leads to an analog of Theorem~\ref{T.gp_main} which,
similarly, has two alternative conclusions.
We shall describe one of these by referring to
that earlier theorem, and spell out the other.

\begin{theorem}\label{T.monoid_main}
Let $B$ be a monoid satisfying the cancellation
condition~\eqref{d.cancel}, and suppose
there exists a family $(M_i)_{i\in I}$ of monoids,
and a monoid homomorphism $f:\prod_{i\in I} M_i\to B$ such that the
induced homomorphism $\pi f:\prod_{i\in I} M_i\to B/U(Z(B))$ does
not factor through any finite product of ultraproducts of the $M_i.$

Then either\\[.2em]
\textup{(a)}~the group $U(B)$ satisfies the hypothesis,
and hence the conclusions, of Theorem~\ref{T.gp_main},\\[.2em]
or\\[.2em]
\textup{(b)}~$B$ contains a family of elements $(a_S)$
indexed by the subsets $S\subseteq\omega,$ such that
\begin{equation}\begin{minipage}[c]{35pc}\label{d.monoid_commute}
$a_\emptyset=e,$ and
all the elements $a_S$ $(S\subseteq\omega)$ commute with one another,
\end{minipage}\end{equation}
\begin{equation}\begin{minipage}[c]{35pc}\label{d.monoid_a_disj}
For disjoint sets $S,\ T\subseteq\omega,$
one has $a_S\,a_T=a_{S\cup T}.$
\end{minipage}\end{equation}
\begin{equation}\begin{minipage}[c]{35pc}\label{d.monoid_subseteq}
For sets $S\subsetneq T\subseteq\omega,$
$a_T$ is a right multiple of $a_S,$
but $a_S$ is not a right multiple of $a_T.$
\end{minipage}\end{equation}
\end{theorem}

\begin{proof}
The two cases of~\eqref{d.monoid_Jn}
will yield the two alternative conclusions shown.
It is easy to verify that~(\ref{d.monoid_Jn}a) yields conclusion~(a).

In case~(\ref{d.monoid_Jn}b), let $L_n=\prod_{i\in J_n} M_i,$
and take elements $z_n\in L_n$ with noninvertible images in $B.$
For each $S\subseteq\omega,$ let $a_S$ be the image under $f$ of
the element of $\prod_{n\in S} L_n\subseteq\prod_{n\in\omega} L_n$
whose $\!n\!$-th coordinate is $z_n$ for each $n\in S.$
Then~\eqref{d.monoid_commute} and~\eqref{d.monoid_a_disj} are
immediate, and the first assertion of~\eqref{d.monoid_subseteq}
follows from~\eqref{d.monoid_a_disj} applied to $S$ and $T-S.$

To get the final assertion of~\eqref{d.monoid_subseteq}, choose
any $n\in T-S,$ and note that by~\eqref{d.monoid_a_disj}, we have
\begin{equation}\begin{minipage}[c]{35pc}\label{d.multiple}
$a_T\ =\ a_S\ a_{T-(S\cup\{n\})}\ a_{\{n\}}.$
\end{minipage}\end{equation}
If we also had $a_S=a_T\,b$ for some $b\in B,$ then
substituting this into the right-hand-side of~\eqref{d.multiple}
and cancelling $a_T$ by~\eqref{d.cancel}, we could conclude
that $a_{\{n\}}$ was left invertible, hence
by the observation following~\eqref{d.cancel},
invertible, contradicting our choice of $z_n.$
\end{proof}

In case~(b) of Theorem~\ref{T.monoid_main},
we cannot say, as we can in case~(a), that distinct
sets $S$ yield distinct elements $a_S\in B.$
For instance, let $B$ be the factor-monoid of
the additive monoid $(\Z^{\geq 0})^\omega$ by the relation
that equates elements $x$ and $y$ if there is
some $n\geq 0$ such that $x$ and $y$ agree at all but the first $n$
coordinates, and such that the sum of
the entries at those first $n$ coordinates
is the same for $x$ and $y.$
Then $B$ is a cancellative abelian monoid with trivial
group of units, and the quotient map $f:(\Z^{\geq 0})^\omega\to B$
does not annihilate any of the $\delta_n$
(defined in $(\Z^{\geq 0})^\omega$ as in $\Z^\omega).$
Hence~(\ref{d.monoid_Jn}b) holds for this $f,$
with the $J_n$ taken to be the singletons $\{n\},$ and $z_n=\delta_n.$
But defining the $a_S$ in terms of these as in
the proof of Theorem~\ref{T.monoid_main},
we find that for finite subsets $S,\,T\subseteq\omega$
of the same cardinality, we have $a_S=a_T$ in $B;$
so the $a_S$ are not all distinct.

Nevertheless, in the situation of Theorem~\ref{T.monoid_main}(b)
we always get continuum many distinct $a_S.$
For by~\eqref{d.monoid_subseteq}, distinct {\em comparable}
sets give distinct elements; and
the partially ordered set of subsets of any countably infinite set
has chains of the order-type of the real numbers.
(Indeed, the countable set of rational numbers has the chain of
Dedekind cuts, and any countably infinite set can be
put in bijective correspondence with the rationals.)
Thus, we get

\begin{corollary}\label{C.monoid_c}
In the situation of conclusion~\textup{(b)} of
Theorem~\ref{T.monoid_main}, $B$ has a set of mutually
commuting noninvertible elements
which form, under the relation of divisibility,
a chain with the order-type of the real numbers.
In particular, $B$ \textup{(}and in fact, $B/U(Z(B)))$
has at least the cardinality of the continuum.\qed
\end{corollary}

The results proved above are far from optimal.
For instance, the conclusions of Theorem~\ref{T.monoid_main} and
Corollary~\ref{C.monoid_c} are consistent with $B$ being
the additive monoid $\mathbb{R}^{\geq 0}$ of nonnegative real numbers;
but that case is easy to exclude.
Indeed, suppose $B=\mathbb{R}^{\geq 0}$ admitted a map as in the
hypothesis of Theorem~\ref{T.monoid_main}.
By the proof of Lemma~\ref{L.monoid_Us}, since $B$
has trivial group of units, we must have a homomorphism
$f:\prod_{n\in\omega} L_n\to B$ such that each $L_n$ has
an element $x_n$ making $f(x_n)$ a positive real number.
By the Archimedean property of the reals, we can
modify our choices of $x_n$ so that
for each $n$ we have $f(x_n)\geq 1.$
Thus, when we construct elements
$a_S\in B$ as in the proof of the theorem, we get
$a_{\{n\}}\geq 1$ for each $n,$ from which
it follows that for any $S\subseteq\omega$ of $\geq m$
elements $(m\in\omega),$ $a_S\geq m.$
For $S$ infinite, this gives a contradiction; so $B$ admits
no such map.
It is not clear to me what the best assertion that can be
gotten by this technique is.

Let us also note that in place of the $\!2\!$-way subdivision of
the sets $J_n$ used in the proof of Lemma~\ref{L.monoid_Us},
we could (at least if we assumed full cancellativity rather than
just~\eqref{d.cancel}) have used a $\!3\!$-way subdivision, noting
that for each $n,$
$f(L_n)$ either contains noncommuting {\em invertible} elements of $B,$
or contains noncommuting {\em non}invertible elements, or
contains a central noninvertible element.
(Cancellativity is needed to show that if a nonunit $x$
and a unit $u$ fail to commute, then so do the two
nonunits $x$ and $xu.)$
So there must be infinitely many $n$ for
which one of these statements holds, and
we can deduce a three-alternative conclusion:
Either, as before, we have invertible elements $a_S,\,b_S\in B$
indexed by the subsets $S\subseteq\omega$
which can be distinguished by their commutativity relations,
or we have elements $a_S,\,b_S\in B$ which, except
for $a_\emptyset,$ $b_\emptyset,$ are {\em noninvertible}, and satisfy
the same relations and can be distinguished in the same way,
or we have {\em central} elements $a_S$ which
satisfy~\eqref{d.monoid_commute}-\eqref{d.monoid_subseteq}.

Though one could define ``almost direct factors'' for
monoids, as for groups, using submonoids that are each others'
centralizers, there doesn't seem
to be an analogous way to ``split'' a monoid
based on noninvertible central
elements; so I have not attempted to formulate
an analog of Theorem~\ref{T.gp_CC}.
I leave further exploration of these
questions to those better versed than I in the study of monoids.

One can also consider for {\em semigroups} the same factorization
properties that we have studied here for monoids.
Since the above constructions involved
elements of direct products defined to have the
value $e$ on complements of given
subsets $S$ of our index set, the absence of identity elements
should lead to changes in what can be proved.

\section{Lattices: a case worth looking at}\label{S.lattices}

One other class of mathematical structures suggests itself, to which
similar methods might be applicable -- lattices.
Just as a direct product decomposition of a group or monoid
leads to certain pairs of elements that must commute,
and a direct product decomposition of a ring leads to certain
pairs of elements that must have zero product,
so a direct product decomposition of a lattice leads to
certain $\!3\!$-tuples of elements that must
satisfy distributivity relations.
Perhaps this observation can be
used to get lattice-theoretic analogs
of some of the results of this note.

(In \cite[\S5]{cap_ultra} I speculate on very
general properties of a variety of algebras that would
allow one to get such results; but I am not confident that %
that approach will go anywhere.)

\section{Acknowledgements}\label{S.ackn}
I am indebted to Thomas Scanlon for suggesting Lemma~\ref{L.TS},
to M.\,Dugas, L.\,Fuchs, K.\,M.\,Rangaswamy and J.\,Rotman for
helpful correspondence regarding \S\S\ref{S.ab}--\ref{S.ab_ac+cotor},
and to the referee for several useful suggestions.

\end{document}